\documentclass[reqno]{amsart}
\usepackage[a4paper]{geometry}
\usepackage[T1]{fontenc}
\usepackage[utf8]{inputenc}
\usepackage{amssymb}
\usepackage[mathscr]{eucal}
\usepackage{mathtools}
\usepackage{enumerate}
\usepackage{color}
\usepackage[colorlinks=true]{hyperref}
\hypersetup{urlcolor=blue,citecolor=red,linkcolor=blue}
\usepackage[initials]{amsrefs}
\numberwithin{equation}{section}
\numberwithin{figure}{section}
\theoremstyle{plain}
\newtheorem{theorem}{Theorem}[section]
\newtheorem{proposition}[theorem]{Proposition}
\theoremstyle{definition}
\newtheorem*{rh-pb*}{Basic RH problem}
\theoremstyle{remark}
\newtheorem{remark}[theorem]{Remark}
\newtheorem*{back*}{From $M^R$ back to $\tilde u$}
\newtheorem*{notations*}{Notations}
\newtheorem*{stepone*}{Step \rm{(i)}}
\newtheorem*{steptwo*}{Step \rm{(ii)}}
\newtheorem*{stepthree*}{Step \rm{(iii)}}
\newtheorem*{stepfour*}{Step \rm{(iv)}}
\newtheorem*{stepfive*}{Step \rm{(v)}}
\providecommand{\D}[1]{\mathbb{#1}}
\providecommand{\C}[1]{\mathcal{#1}}
\newcommand{\dd}{\mathrm{d}}
\newcommand{\eul}{\mathrm{e}}
\newcommand{\ii}{\mathrm{i}}

\providecommand{\accol}[1]{\lbrace#1\rbrace}

\renewcommand{\Im}{\operatorname{Im}}
\renewcommand{\Re}{\operatorname{Re}}
\newcommand{\ord}{\mathrm{O}}
\newcommand{\osmall}{\mathrm{o}}
\DeclareMathOperator{\Res}{Res}
\begin{document}
\title[Large-time asymptotics for the mCH equation]{The modified Camassa-Holm equation on a nonzero background: large-time asymptotics for the Cauchy problem}
\author[A.~Boutet de Monvel]{Anne Boutet de Monvel}
\address{Institut de Math\'ematiques de Jussieu-Paris Rive Gauche,
Universit\'e de Paris,
8 place Aur\'elie Nemours, case 7012,
75205 Paris Cedex 13,
France}
\email{anne.boutet-de-monvel@imj-prg.fr}
\author[I.~Karpenko]{Iryna Karpenko}
\address{B.I.~Verkin Institute for Low Temperature Physics and Engineering,
47 Nauky Avenue, 61103 Kharkiv, Ukraine}
\email{inic.karpenko@gmail.com}
\author[D.~Shepelsky]{Dmitry Shepelsky}
\address{B.I.~Verkin Institute for Low Temperature Physics and Engineering,
47 Nauky Avenue, 61103 Kharkiv, Ukraine\\
V.N.~Karazin Kharkiv National University, 4 Svobody Square, 61022 Kharkiv, Ukraine}
\email{shepelsky@yahoo.com}
\begin{abstract}
This paper deals with the Cauchy problem for the modified Camassa-Holm (mCH) equation
\begin{alignat*}{4}
&m_t+\left((u^2-u_x^2)m\right)_x=0,&\quad&m\coloneqq u-u_{xx},&\quad&t>0,&\;&-\infty<x<+\infty,\\
&u(x,0)=u_0(x),&&&&&&-\infty<x<+\infty,
\end{alignat*}
in the case when the initial data $u_0(x)$ as well as the solution $u(x,t)$ are assumed to approach a nonzero constant as $x\to\pm\infty$. In a recent paper we developed the Riemann--Hilbert formalism for this problem, which allowed us to represent the solution of the Cauchy problem in terms of the solution of an associated Riemann--Hilbert factorization problem. In this paper, we apply the nonlinear steepest descent method, based on this Riemann--Hilbert formalism, to study the large-time asymptotics of the solution of this Cauchy problem. We present the results of the asymptotic analysis in the solitonless case for the two sectors $\frac{3}{4}<\frac{x}{t}<1$ and $1<\frac{x}{t}<3$ (in the $(x,t)$ half-plane, $t>0$), where the leading asymptotic term of the deviation of the solution from the background is nontrivial: this term is given by modulated (with parameters depending on $\frac{x}{t}$), decaying (as $t^{-1/2}$) trigonometric oscillations.
\end{abstract}
\maketitle
\section{Introduction}\label{sec:1}

In the present paper, we consider the initial value problem
for the modified Camassa--Holm (mCH) equation:
\begin{subequations}\label{mCH1-ic}
\begin{alignat}{4}           \label{mCH-1-ic}
&m_t+\left((u^2-u_x^2)m\right)_x=0,&\quad&m\coloneqq u-u_{xx},&\quad&t>0,&\;&-\infty<x<+\infty,\\
&u(x,0)=u_0(x),&&&&&&-\infty<x<+\infty\label{IC},
\end{alignat}
\end{subequations}
assuming that $u_0(x)\to 1$ as $x\to\pm\infty$ and that the time evolution preserves this behavior: $u(x,t)\to 1$ as $x\to\pm\infty$ for all $t>0$. We are interested in the study of the behavior of $u(x,t)$ as $t\to+\infty$. 

Equation \eqref{mCH-1-ic} is an integrable modification, with cubic nonlinearity, of the Camassa--Holm (CH) equation \cites{CH93,CHH94}
\begin{equation}\label{CH-1}
m_t+\left(u m\right)_x + u_x m=0,\quad m\coloneqq u-u_{xx}.
\end{equation}
The  Camassa--Holm equation has been studied intensively over the two decades, due to its rich mathematical structure as well as applications for modeling the unidirectional propagation of shallow water waves over a flat bottom \cites{J02,CL09}. The CH and mCH equations are both integrable in the sense that they have Lax pair representations, which allows to develop the inverse scattering method, in one form or another, to study the properties of solutions of initial (Cauchy) and initial boundary value problems for these equations. In particular, the inverse scattering method in the form of a Riemann--Hilbert (RH) problem developed for the CH equation with linear dispersion \cite{BS08} allowed to study the large-time behavior of solutions of initial as well as initial boundary value problems for the CH equation \cites{BS08-2,BKST09,BS09,BIS10} using the (appropriately adapted) nonlinear steepest descent method \cite{DZ93}.

Over the last few years various modifications and generalizations of the CH equation have been introduced, see, e.g., \cite{YQZ18} and references therein. Novikov \cite{N09} applied the perturbative symmetry approach in order to classify integrable equations of the form
\[
\left(1-\partial_x^2\right) u_t=F(u, u_x, u_{xx}, u_{xxx}, \dots),\qquad u=u(x,t), \quad \partial_x=\partial/\partial x,
\]
assuming that $F$ is a homogeneous differential polynomial over $\D{C}$, quadratic or cubic in $u$ and its $x$-derivatives (see also \cite{MN02}). In the list of equations presented in \cite{N09}, equation (32), which was the second equation with \emph{cubic} nonlinearity, had the form \eqref{mCH-1-ic}. In an equivalent form, this equation was given by Fokas in \cite{F95} (see also \cite{OR96} and \cite{Fu96}). Shiff \cite{S96} considered equation \eqref{mCH-1-ic} as a dual to the modified Korteweg--de Vries equation (mKdV) and introduced a Lax pair for \eqref{mCH-1-ic} by rescaling the entries of the spatial part of a Lax pair for the mKdV equation. An alternative (in fact, gauge equivalent) Lax pair for \eqref{mCH-1-ic} was given by Qiao \cite{Q06}, so the mCH equation is also referred to as the Fokas--Olver--Rosenau--Qiao (FORQ) equation \cite{HFQ17}.

Equation \eqref{mCH-1-ic} belongs to the class of peakon equations: it has solutions in the form of localized, peaked traveling waves -- peakons \cite{GLOQ13}. The dynamical stability of peakons is discussed in \cite{QLL13} (see also \cite{LLOQ14} for the stability of peakons of a generalized mCH equation). Multipeakon solutions are discussed in \cite{CS18} using the inverse spectral method for an associated peakon system of ordinary differential equations. 

The local well-posedness and wave-breaking mechanisms for the mCH equation and its generalizations, particularly, the mCH equation with linear dispersion, are discussed in \cites{GLOQ13,FGLQ13,LOQZ14,CLQZ15,CGLQ16}. Algebro-geometric quasiperiodic solutions are studied in \cite{HFQ17}. The local well-posedness for classical solutions and global weak solutions to \eqref{mCH-1-ic} in Lagrangian coordinates are discussed in \cite{GL18}.

The Hamiltonian structure and Liouville integrability of peakon systems are discussed in \cites{AK18,OR96,GLOQ13,CS17}. In \cite{K16}, a Liouville-type transformation was presented relating the isospectral problems for the mKdV equation and the mCH equation, and a Miura-type map from the mCH equation to the CH equation was introduced. The B\"acklund transformation for the mCH equation and a related nonlinear superposition formula are presented in \cite{WLM20}.

In the case of the CH equation, the inverse scattering transform method (particularly, in the form of a Riemann--Hilbert factorization problem) works for the version of this equation, considered for functions decaying at spatial infinity, with a linear dispersion term added to \eqref{CH-1} or, equivalently, when \eqref{CH-1} is considered on a nonzero background. This is because the inverse scattering method requires that the spatial equation from the Lax pair associated to the CH equation have continuous spectrum. On the other hand, the asymptotic analysis of the dispersionless CH equation \eqref{CH-1} on zero background (where the spectrum is purely discrete) requires a different tool (although having a certain analogy with the Riemann--Hilbert method), namely the analysis of a coupling problem for entire functions \cites{E19,ET13,ET16}.

In the case of the mCH equation, the situation is similar: the inverse scattering method for the Cauchy problem can be developed when equation \eqref{mCH-1-ic} is considered on a nonzero background. The Riemann--Hilbert formalism for this problem has been developed in \cite{BKS20}.

In the present paper, we study the large-time behavior of the solution
of the Cauchy problem for the mCH equation on a nonzero background \eqref{mCH1-ic}, taking the formalism developed in \cite{BKS20} as the starting point. Focusing on the solitonless case, in Section \ref{sec:2} we reduce the original (singular) RH problem representation for the solution of \eqref{mCH1-ic} to the resolution of a regular RH problem. Then, in Section \ref{s:as-reg}, the latter problem is analyzed asymptotically, as $t\to+\infty$. We finally obtain the leading asymptotic terms for the solution of the Cauchy problem \eqref{mCH1-ic}, in the two sectors of the $(x,t)$ half-plane, $1<\frac{x}{t}<3$ and $\frac{3}{4}<\frac{x}{t}<1$ where the deviation from the background value is nontrivial. In those sectors this deviation exhibits slowly decaying (of order $t^{-1/2}$), modulated (by $\frac{x}{t}$) oscillations (Theorems~\ref{th2} and \ref{th4}), while in the remaining sectors $\frac{x}{t}>3$ and $\frac{x}{t}<\frac{3}{4}$ it decays rapidly to $0$.

\begin{notations*}
Furthermore, $\sigma_1\coloneqq\left(\begin{smallmatrix}0&1\\1&0\end{smallmatrix}\right)$ and $\sigma_3\coloneqq\left(\begin{smallmatrix}1&0\\0&-1\end{smallmatrix}\right)$ denote the standard Pauli matrices. We let $\D{C}_+=\accol{\Im\mu>0}$ and $\D{C}_-=\accol{\Im\mu<0}$ denote the open upper and lower complex half-planes. We also let $f^*(\mu)\coloneqq\overline{f(\bar\mu)}$ denote the Schwarz conjugate of a function $f(\mu)$, $\mu\in\D{C}$. If $M$ is a $2\times 2$ matrix we denote by $M^{(1)}$ and $M^{(2)}$ its first and second columns, respectively. 
\end{notations*}

\section{Reduction to a regular RH problem}\label{sec:2}

Introducing a new function $\tilde u$ by
\begin{equation}\label{utilde}
u(x,t)=\tilde u(x-t,t)+1,
\end{equation}
the mCH equation \eqref{mCH-1-ic} reduces to
\begin{subequations}\label{mCH2}
\begin{align}\label{mCH-2}
&\tilde m_t+\left(\tilde\omega\tilde m\right)_x=0,\\
\label{tm}
&\tilde m\coloneqq\tilde u-\tilde u_{xx}+1,\\
\label{tom}
&\tilde\omega\coloneqq\tilde u^2-\tilde u_x^2+2\tilde u,
\end{align}
\end{subequations}
where the solution $\tilde u$ is considered on zero background: $\tilde u(x,t)\to 0$ as $x\to\pm\infty$ for all $t\geq 0$. The Riemann--Hilbert (RH) approach for the Cauchy problem for equation \eqref{mCH2} has recently been developed in \cite{BKS20}. This resulted in a parametric representation for $\tilde u(x,t)$ in terms of the solution of an appropriate RH problem proposed in \cite{BKS20},
according to the following algorithm:
\begin{enumerate}[(a)]
\item
Given $u_0(x)$, construct the ``reflection coefficient'' $r(\mu)$, $\mu \in\D{R}$ and, if applicable, the ``discrete spectrum data'' $\{\mu_j,\rho_j\}_{j=1}^N$,	by solving the Lax pair equations associated with \eqref{mCH2}, whose coefficients are determined in terms of $u_0(x)$.
\item
Construct the jump matrix $J(y,t,\mu)$, $\mu\in\D{R}$ by
\begin{equation}\label{J-J0}
J(y,t,\mu)\coloneqq\eul^{-p(y,t,\mu)\sigma_3}J_0(\mu)\eul^{ p(y,t,\mu)\sigma_3}
\end{equation}
where
\begin{equation}\label{p-y}
p(y,t,\mu)\coloneqq -\frac{\ii(\mu^2-1)}{4\mu}\left(-y+\frac{8\mu^2}{(\mu^2+1)^2}t\right)
\end{equation}
and
$J_0(\mu)$ is defined by
\begin{equation}\label{J0}
J_0(\mu)\coloneqq\begin{pmatrix}
1-r(\mu)r^*(\mu)&r(\mu)\\-r^*(\mu)&1
\end{pmatrix}.
\end{equation}
\item
Solve the following \textbf{RH problem} (parametrized by $y$ and $t$):
Find a piece-wise (w.r.t.~$\D{R}$) meromorphic (in the complex variable $\mu$), $2\times 2$-matrix valued function $M(y,t,\mu)$ satisfying the following conditions:
\begin{enumerate}[\textbullet]
\item
The jump condition
\begin{equation}\label{jump-y}
M_+(y,t,\mu)=M_-(y,t,\mu) J(y,t,\mu),\qquad\mu\in\D{R},\quad\mu\neq\pm 1.
\end{equation}
\item
The residue conditions
\begin{equation}\label{res-M-hat}
\begin{split}
\Res_{\mu_j}M^{(1)}(y,t,\mu)&=\frac{1}{\varkappa_j(y,t)} M^{(2)}(y,t,\mu_j),\\
\Res_{\bar\mu_j} M^{(2)}(y,t,\mu)&=\frac{1}{\overline{\varkappa_j}(y,t)} M^{(1)}(y,t,\overline{\mu_j}),
\end{split}
\end{equation}
with $\varkappa_j(y,t)\coloneqq\rho_j\eul^{-2p(y,t,\mu_j)}$.
\item
The normalization condition
\begin{equation}\label{normalization}
M(y,t,\mu)\to I\text{ as }\mu\to\infty.
\end{equation}
\item
The symmetries
\begin{equation}\label{sym-M-hat}
M(\mu)=\overline{M(\bar\mu^{-1})}=\sigma_3 \overline{M(-\bar\mu)}\sigma_3=\sigma_1\overline{M(\bar\mu)}\sigma_1,
\end{equation}
where $M(\mu)\equiv M(y,t,\mu)$.
\item
The singularity conditions
\begin{subequations}\label{sing-M-hat}
\begin{alignat}{3}\label{sing-M-hat-a}
 M(y,t,\mu)&=\frac{\ii\alpha_+(y,t)}{2(\mu -1)}\begin{pmatrix} -c & 1 \\ -c & 1 \end{pmatrix}+\ord(1)&\quad&\text{as }\mu\to 1,&\;&\Im\mu>0,\\
\label{sing-M-hat-b}
 M(y,t,\mu)&=-\frac{\ii\alpha_+(y,t)}{2(\mu+1)}\begin{pmatrix}c&1\\ -c & -1 \end{pmatrix}+\ord(1)&&\text{as }\mu\to -1,&&\Im\mu> 0,
\end{alignat}
\end{subequations}
where $c=1+r(1)$ (generically, $c=0$) whereas $\alpha_+(y,t)\in\D{R}$ is not specified.
\end{enumerate}
\item
Having found the solution $M(y,t,\mu)$ of this RH problem (which is unique, if it exists, see \cite{BKS20}), extract the real-valued functions $a_j(y,t)$, $j=1,2,3$ from the expansion of $M(y,t,\mu)$ at $\mu=\ii$:
\begin{equation}\label{hat-M-i}
M(y,t,\mu)=\begin{pmatrix}
                  a_1(y,t) & 0 \\
                  0 & a_1^{-1}(y,t)
\end{pmatrix}
+\begin{pmatrix}
                  0 & a_2(y,t) \\
                  a_3(y,t) & 0
\end{pmatrix}(\mu-\ii)+\ord((\mu-\ii)^2), \quad \mu\to\ii.
\end{equation}
\item
Obtain $\tilde u(x,t)$ in parametric form as follows:
\[
\tilde u(x,t)=\hat u(y(x,t),t),
\]
where
\begin{equation}\label{u-and-x}
\begin{split}
\hat u(y,t)&=-a_2(y,t)a_1(y,t)-a_3(y,t)a_1^{-1}(y,t),\\
x(y,t)&=y + 2\ln a_1(y,t).
\end{split}
\end{equation}
\end{enumerate}

\begin{remark}
To simplify notations in this paper, compared to \cite{BKS20}, we have removed the symbol ``hat'' over many functions (e.g., $M(y,t,\mu)$, $\alpha_+(y,t)$, etc.). Another difference is that $M_+$ and $M_-$ are exchanged in the jump relation \eqref{jump-y} so that here the jump is the inverse of that in \cite{BKS20}: $J_0=\hat J_0^{-1}$ and $J=\hat J^{-1}$.
\end{remark}

\begin{remark}
The symmetries \eqref{sym-M-hat} are consistent with the symmetries of $r(\mu)$
\begin{equation}\label{sym-r}
r(\mu)=-\overline{r(-\mu)}=\overline{r(\mu^{-1})}
\end{equation}
and the invariance of the set $\{\mu_j,\rho_j\}_{j=1}^N$:
$-\overline{\mu_j}=\mu_{j'}$ and $-\mu_j^{-1}=\mu_{j''}$ with $\rho_j=\overline{\rho_{j'}}=-\mu_j^{-2}\rho_{j''}$.
These symmetries and invariances follow from the construction of the RH problem above in terms of the dedicated (Jost) solutions of the Lax pair equations associated with the mCH equation, see \cite{BKS20}. Moreover, the symmetries \eqref{sym-M-hat} imply the particular structure of the matrices in \eqref{hat-M-i}.
\end{remark}

In the general context of nonlinear integrable equations,
the RH problem formalism (i.e., the representation of the solution
of the original problem --- the Cauchy problem for a nonlinear integrable PDE ---
in terms of the solution of an associated RH problem)
 allows reducing the problem of the large
time analysis of the solution of the nonlinear PDE to that of the RH problem.
Residue conditions (if any) involved in the  RH problem formulation
generate a soliton-type, non-decaying contribution to the asymptotics
whereas the jump conditions are responsible for  the dispersive (decaying) part,
details of which can be retrieved applying an appropriate modification
of the nonlinear steepest descent method to the asymptotic analysis of
a preliminarily regularized RH problem (i.e., a RH problem involving the jump and normalization conditions only).

With this respect we notice that the residue conditions \eqref{res-M-hat} can be handled in a standard way: either adding to the contour small circles around each $\mu_j$ and $\bar \mu_j$ and reducing the residue conditions to associated jump conditions across the circles or using the Blaschke--Potapov factors (see, e.g., \cite{BKST09}); in both approaches, the original RH problem is reduced to a RH problem without residue conditions.

As for the singularity conditions, we notice that in the case of the Camassa--Holm equation, where such a condition is also involved in the matrix RH problem formalism, an efficient way to handle it is to reduce the matrix RH problem to a vector one, multiplying from the left by the constant vector $(1,1)$. Indeed, the singularity condition for the CH equation has the form of \eqref{sing-M-hat-b}, and thus this multiplication ``kill'' the singularity, reducing the RH problem to a regular one. With this respect, we notice that the matrix RH problem for the modified Camassa--Holm equation is different: it also involves the singularity condition \eqref{sing-M-hat-a}, which, obviously, cannot be removed using the same trick.

In the present paper, we focus on the study of the dispersive part of the large-time asymptotics of solutions of the Cauchy problems for the mCH equation. Accordingly, we proceed with the solitonless case assuming that there are no residue conditions (inclusion of the discrete spectrum can then be made following a well-developed technique, see, e.g., \cite{BKST09}).

In this section we  reduce the original RH problem (which is still singular due to conditions \eqref{sing-M-hat}) to a regular one, proceeding in two steps.

In Step 1, we reduce the RH problem with the singularity conditions
\eqref{sing-M-hat} at $\mu=\pm 1$ to a RH problem
which is characterized by the following two conditions:
\begin{enumerate}[(i)] 
\item
the matrix entries are regular at $\mu=\pm 1$, but the  determinants of the (matrix) solution vanish at $\mu=\pm 1$ (notice that $\det M(\mu)\equiv 1$ for the solution of the original RH problem);
\item
the solution is singular at $\mu=0$.
\end{enumerate}

Then, in Step 2, the latter RH problem  is reduced to a regular one, i.e., to a RH problem with the jump and normalization conditions only.

\begin{proposition}\label{p1}
Let $M(y,t,\mu)$ be a solution of the RH problem \eqref{jump-y}, \eqref{normalization}--\eqref{sing-M-hat}. Define $\tilde M$ by
\begin{equation}\label{M-to-tilde}
\tilde M(y,t,\mu)\coloneqq\left(I-\frac{1}{\mu}\sigma_1\right)M(y,t,\mu).
\end{equation}
Then $\tilde M(\mu)\equiv\tilde M(y,t,\mu)$ is the unique solution of the following RH problem:
\begin{enumerate}[\rm({C}1)]
\item
$\tilde M(\mu)$ is analytic in $\D{C}_+$ and $\D{C}_-$ and continuous up to $\D{R}\setminus\{0\}$.
\item
$\tilde M(\mu)$ satisfies the jump condition \eqref{jump-y} with the jump defined by \eqref{J-J0}--\eqref{J0}.
\item
$\tilde M(\mu)\to I$ as $\mu\to\infty$.
\item
$\tilde M(\mu)=-\frac{1}{\mu}\sigma_1+\ord(1)$ as $\mu\to 0$.
\item
$\det\tilde M(\pm 1)=0$.
\item
$\tilde M(\mu^{-1})=-\mu\tilde M(\mu)\sigma_1$.
\end{enumerate}
\end{proposition}

\begin{proof}
First, let's check that $\tilde M(y,t,\mu)$ constructed from $M(y,t,\mu)$ satisfies the conditions above. The limiting properties (C3) and (C4) as $\mu\to \infty$ and as $\mu\to 0$ are obviously satisfied (by construction) whereas (C2) results from the fact that a multiplication from the left does not change the jump conditions. Further, since $\det M(y,t,\mu)\equiv 1$, it follows that $\det\tilde M(y,t,\mu)=1-\frac{1}{\mu^2}$ and thus $\det\tilde M(y,t,\pm 1)=0$. Moreover, as $\mu\to 1$ we have
\begin{align*}
\left(\tilde M_{11}(\mu),\tilde M_{12}(\mu)\right) & =
\left(M_{11}(\mu), M_{12}(\mu)\right)-\frac{1}{\mu}\left(M_{21}(\mu), M_{22}(\mu)\right) \\
&=\left(M_{11}(\mu) - M_{21}(\mu), M_{12}(\mu)-M_{22}(\mu)\right) +\ord(1)=\ord(1)
\end{align*}
due to \eqref{sing-M-hat-a}. Similarly, as $\mu\to -1$ we have
\[
\left(\tilde M_{11}(\mu), \tilde M_{12}(\mu)\right)=\left(M_{11}(\mu) + M_{21}(\mu), M_{12}(\mu)+M_{22}(\mu)\right) +\ord(1)=\ord(1)
\]
due to \eqref{sing-M-hat-b}. Similarly for $\bigl(\tilde M_{21}(\mu), \tilde M_{22}(\mu)\bigr)$; thus $\tilde M(y,t,\mu)$ is non-singular at $\mu=\pm 1$. Finally, (C6) follows from the symmetry relations \eqref{sym-M-hat} (more precisely, from $M(\mu^{-1})=\sigma_1 M(\mu) \sigma_1$).

Now, let's prove that the solution of the RH problem (C1)--(C6) above
is unique (if exists). First, we notice that if $\tilde M(y,t,\mu)$ solves the RH problem (C1)--(C6), then
\begin{equation}\label{det-M}
\det \tilde M(y,t,\mu)=1-\frac{1}{\mu^2}.
\end{equation}
Indeed, since $\det J(y,t,\mu)\equiv 1$ and $\det M(y,t,\mu)$ is bounded at $\mu=\infty$, it follows that $\det M(\mu)$ is a rational function. Moreover, from (C4) we have that $\det M(\mu)=-\frac{1}{\mu^2} + \frac{c}{\mu} + \ord(1)$ as $\mu\to 0$, with some $c\equiv c(y,t)$. Taking into account (C3) we have that $\zeta(y,t,\mu)\coloneqq\det M(y,t,\mu) - 1 + \frac{1}{\mu^2} - \frac{c}{\mu}$ is a bounded entire function of $\mu$, which, by Liouville's theorem and (C3), vanishes for all $(y,t)$. Finally, evaluating $\zeta(y,t,\mu)$ at $\mu=\pm 1$ and using (C5), it follows that $c(y,t)\equiv 0$ and thus \eqref{det-M} follows.

Now let's assume that $\tilde{\tilde M}$ is another solution of the RH problem (C1)--(C6) and define $N(\mu)\coloneqq\tilde M(\mu){\tilde{\tilde M}}^{-1}(\mu)$. Since $\tilde M$ and $\tilde{\tilde M}$ satisfy the same jump conditions, $N(\mu)$ is a rational function, with possible singularities at $\mu=0, -1, 1$. In view of \eqref{det-M} and (C3), ${\tilde{\tilde M}}^{-1}(\mu)=\frac{\mu^2}{\mu^2 -1}(\frac{1}{\mu} \sigma_1 + \ord(1))= \ord(\mu)$ as $\mu\to 0$ and thus  $N(\mu)$ is non-singular at $\mu=0$. In order to prove that $N(\mu)$ is non-singular at $\mu=\pm 1$, we use relation (C6). In particular, we have $\tilde M(1)=-\tilde M(1)\sigma_1$ and thus $\tilde M(\mu)=\left(\begin{smallmatrix}g_1 &-g_1\\g_2&-g_2\end{smallmatrix}\right)+\ord(\mu-1)$ as $\mu\to 1$, with some $g_j$, $j=1,2$. Consequently, ${\tilde{\tilde M}}^{-1}(\mu)=\frac{\mu^2}{\mu^2 -1}\left(\left(\begin{smallmatrix}-\tilde{g}_2 & \tilde{g}_1 \\ -\tilde{g}_2 & \tilde{g}_1\end{smallmatrix}\right)+\ord(\mu-1)\right)$ as $\mu\to 1$, with some $\tilde{g}_j$, $j=1,2$, which implies that $N(\mu)$ is bounded as $\mu\to 1$. Similarly for $\mu\to -1$. Therefore, $N(\mu)$ is an entire function such that $N(\infty)=I$ and thus $N(\mu)\equiv I$ by Liouville's theorem.
\end{proof}

\begin{remark}
Assuming $r(\mu)=-\overline{r(-\mu)}$ (see \eqref{sym-r}), we have that
$J(\mu)$ satisfies the symmetries
\[
J(\mu)=\sigma_3 \overline{J(-\mu)}\sigma_3=\sigma_1 \overline{J^{-1}(\mu)} \sigma_1,
\]
which, due to uniqueness, imply the symmetries for $\tilde M$ similar to those
for $M$:
\begin{equation}\label{sym-tilde-M}
\tilde M(\mu)=\sigma_3\overline{\tilde M(-\bar\mu)}\sigma_3=\sigma_1 \overline{\tilde M(\bar\mu)}\sigma_1
\end{equation}
(taking also into account that the symmetries \eqref{sym-tilde-M}
are consistent with all conditions in the RH problem in Proposition \ref{p1}).
\end{remark}

Step 2 in the reduction of the RH problem is formulated in the following proposition (see \cites{IU86,V00,V03} for the case of the nonlinear Schr\"odinger equation with ``finite density'' boundary conditions).

\begin{proposition}\label{p2}
The solution $\tilde M$ of the RH problem from Proposition \ref{p1} can be represented in terms of the solution of a regular RH problem as follows:
\begin{equation}\label{to-M-reg}
\tilde M(y,t,\mu)=\left(I-\frac{1}{\mu}\Delta(y,t)\right)M^R(y,t,\mu),
\end{equation}
where $M^R(\mu)\equiv M^R(y,t,\mu)$ is the solution of the following RH problem: Find $M^R(\mu)$ such that
\begin{enumerate}[\rm({R}1)]
\item
$M^R(\mu)$ is analytic in $\D{C}_+$ and $\D{C}_-$ and continuous up to the real axis.
\item
$M^R(\mu)$ satisfies the jump condition \eqref{J-J0}--\eqref{jump-y}.
\item
$M^R(\mu)\to I$ as $\mu\to\infty$.
\end{enumerate}
Here $\Delta$ in \eqref{to-M-reg} is expressed in terms of the solution $M^R$ of the RH problem above by: 
\[
\Delta(y,t)=\sigma_1[M^R(y,t,0)]^{-1}.
\]
\end{proposition}

\begin{proof}
Let $M^R(\mu)$ be the solution of the regular RH problem (R1)-(R3) above. Then $\tilde M(y,t,\mu)$ defined by \eqref{to-M-reg} obviously (by construction) satisfies conditions (C1)-(C4) of the RH problem from Proposition \ref{p1}. In order to check conditions (C5) and (C6), we use the matrix structure of $\Delta$ that follows from the symmetries of $M^R(\mu)$.

(i) Since $M^R(\mu)$ and $M(\mu)$ satisfy the same jump conditions, the uniqueness of the solution of the regular RH problem implies that  $M^R(\mu)$ satisfies the same symmetries (see \eqref{sym-M-hat}) (generated by the symmetry $r(\mu)=-\overline{r(-\mu)}$):
\begin{equation}\label{sym-M-reg}
M^R(\mu)=\sigma_3 \overline{M^R(-\bar\mu)}\sigma_3 =
\sigma_1\overline{M^R(\bar\mu)}\sigma_1.
\end{equation}
Considering this for $\mu=0$ it follows that $M^R(y,t,0)=\left(\begin{smallmatrix}\alpha(y,t)&\ii\beta(y,t)\\-\ii\beta(y,t)&\alpha(y,t)\end{smallmatrix}\right)$ with some $\alpha(y,t)\in\D{R}$ and $\beta(y,t)\in\D{R}$. Moreover, $\alpha^2(y,t)-\beta^2(y,t)\equiv 1$ since $\det M^R(\mu)\equiv 1$. Consequently, $\Delta(y,t)$ has the structure
\begin{equation}\label{Delta}
\Delta=\begin{pmatrix}
\ii\beta & \alpha \\ \alpha & -\ii\beta
\end{pmatrix}\text{ with }\alpha^2-\beta^2= 1
\end{equation}
and thus $\det(I-\mu^{-1}\Delta(y,t))=1-\frac{\alpha^2-\beta^2}{\mu^2}=1-\frac{1}{\mu^2}$, which implies (C5). Notice that $\Delta^2\equiv I$.

(ii) Now consider the symmetry $\mu\mapsto\mu^{-1}$. From $r(\mu)=\overline{r(\mu^{-1})}$ it follows that $J(\mu)=\sigma_1 J^{-1}(\mu^{-1})\sigma_1$ and thus $\check M(\mu)\coloneqq\sigma_1 M^R(\mu^{-1})\sigma_1$ satisfies the same jump condition as $M^R(\mu)$ does. Taking into account that $\check M(\infty) =\sigma_1 M^R(0)\sigma_1$, Liouville's theorem implies that $\check M^{-1}(\infty)\check M(\mu)\equiv \sigma_1 [M^R(0)]^{-1} M^R(\mu^{-1})\sigma_1=M(\mu)$, or, in terms of $\Delta$,
\begin{equation}\label{sym-M-R}
M^R(\mu^{-1})=\Delta M^R(\mu) \sigma_1.
\end{equation}
Now, combining \eqref{to-M-reg} with \eqref{sym-M-R} we can express $\tilde M(\mu^{-1})$ in terms of $\tilde M(\mu)$ as follows:
\begin{equation}\label{tilde-M-1}
\tilde M(\mu^{-1})=(I-\Delta \mu)M^R(\mu^{-1})= (I-\Delta \mu)  \Delta M^R(\mu)\sigma_1= Q(\mu)\tilde M(\mu)\sigma_1
\end{equation}
with
\[
Q(\mu)=(I-\Delta \mu) \Delta \left(I-\Delta \mu^{-1}\right)^{-1}.
\]
Using \eqref{Delta}, direct calculations give $Q(\mu)=-\mu I$
and thus the symmetry \eqref{sym-M-R} takes the form of (C6) in Proposition \ref{p1}.
\end{proof}

\subsubsection*{From $M^{R}$ back to $\tilde u$}
Now, we can obtain a parametric representation of the solution $\tilde u(x,t)$ of the Cauchy problem \eqref{mCH2} in terms of the solution $M^R(y,t,\mu)$ of the regular RH problem from Proposition \ref{p2}. First, using \eqref{M-to-tilde} and \eqref{to-M-reg}, we get $M$ from $M^R$:
\begin{equation}\label{M-all-1}
M(\mu)=\left(I-\frac{1}{\mu}\sigma_1\right)^{-1}\left(I-\frac{1}{\mu}\Delta\right)M^R(\mu).
\end{equation}
Then, by \eqref{hat-M-i} and \eqref{u-and-x} we find
\[
M(y,t,\mu)\leadsto\accol{a_1(y,t),a_2(y,t),a_3(y,t)}\leadsto\accol{\hat u(y,t),x(y,t)},
\]
and finally $\tilde u(x,t)=\hat u(y(x,t),t)$.

\section{Large-time asymptotics of the regular RH problem}\label{s:as-reg}

In this section, we study the large-time asymptotics of the solution $M^R(y,t,\mu)$ of the regular RH problem from Proposition~\ref{p2} using the ideas and tools of the nonlinear steepest descent method \cite{DZ93}. The method consists in successive transformations of the original RH problem, in order to reduce it to an explicitly solvable problem. The different steps include appropriate triangular factorizations of the jump matrix; ``absorption'' of the triangular factors with good large-time behavior; reduction, after rescaling, to a RH problem which is solvable in terms of certain special functions; analysis of the approximation errors. Here we focus on deriving the leading terms of the large-time asymptotics, while for error estimates we refer to \cite{L15}.
\subsection{Transformations of the regular RH problem}

Introduce
\[
\theta(\mu,\xi)\coloneqq\hat\theta(k(\mu),\xi),
\]
where
\begin{equation}\label{xi-k-theta}
\xi\coloneqq\frac{y}{t},\quad k(\mu)\coloneqq\frac{1}{4}\left(\mu-\frac{1}{\mu}\right),\quad\hat\theta(k,\xi)\coloneqq k\xi-\frac{2k}{1+4k^2}.
\end{equation}
Hence, $p(y,t,\mu)=\ii t\theta(\mu,\xi)$. The jump matrix \eqref{jump-y} with \eqref{J-J0}--\eqref{J0} allows two triangular factorizations:
\begin{subequations}\label{triang}
\begin{alignat}{3}\label{triang-a}
J(y,t,\mu)&=\begin{pmatrix}
1&  r(\mu)\eul^{-2\ii t\theta}\\
0& 1\\
\end{pmatrix}
\begin{pmatrix}
1& 0\\
-r^*(\mu)\eul^{2\ii t\theta}& 1\\
\end{pmatrix},\\
\label{triang-b}
J(y,t,\mu)&=
\begin{pmatrix}
1& 0\\
-\frac{r^*(\mu)}{1-|r(\mu)|^2}\eul^{2\ii t\theta}& 1\\
\end{pmatrix}
\begin{pmatrix}
1-|r(\mu)|^2& 0\\
0& \frac{1}{1-|r(\mu)|^2}\\
\end{pmatrix}
\begin{pmatrix}
1& \frac{r(\mu)}{1-|r(\mu)|^2}\eul^{-2\ii t\theta}\\
0& 1\\
\end{pmatrix}.
\end{alignat}
\end{subequations}

Following the basic idea of the nonlinear steepest descent method \cite{DZ93}, the factorizations \eqref{triang} can be used in such a way that the (oscillating) jump matrix on $\D{R}$ for a modified RH problem reduces (see the RH problem for $M^{(2)}$ below) to the identity matrix whereas the arising jumps outside $\D{R}$ are exponentially small as $t\to+\infty$. The use of one or another form of the factorization is dictated by the ``signature table'' for $\theta$, i.e., the distribution of signs of $\Im\theta(\mu,\xi)$ (that depends on $\xi$) in the $\mu$-complex plane. The factorization \eqref{triang-a} is appropriate for the (open) intervals of $\D{R}$ (let us denote their union by $\Sigma_a\equiv\Sigma_a(\xi)$) for which $\Im\theta(\mu)$ is positive for $\mu\in\D{C}_+$ close to these intervals (and negative for $\mu\in\D{C}_-$ close to the same intervals). On the other hand the factorization \eqref{triang-b} is appropriate for the (open) intervals of $\D{R}$ (we denote their union by $\Sigma_b(\xi)=\D{R}\setminus\overline{\Sigma_a(\xi)}$), for which $\Im\theta(\mu)$ is negative for $\mu\in\D{C}_+$ close to these intervals.

In turn, one can get rid of the diagonal factor in \eqref{triang-b} using the solution of the following scalar RH problem: Find a scalar function $\delta(\mu,\xi)$ ($\xi$ being a parameter) analytic in
$\D{C}\setminus\overline{\Sigma_b(\xi)}$ such that
\begin{subequations}\label{delta-eq}
\begin{alignat}{3}\label{delta-eq-a}
\delta_+(\mu,\xi)&=\delta_-(\mu,\xi)(1-|r(\mu)|^2),  \quad \mu\in\Sigma_b(\xi),\\
\label{delta-eq-b}
\delta(\mu,\xi)&\to 1,\quad\mu\to\infty.
\end{alignat}
\end{subequations}
The solution of the RH problem \eqref{delta-eq} is given by the Cauchy integral:
\begin{equation}\label{delta}
\delta(\mu,\xi)=\exp\left\{\frac{1}{2\pi\ii}\int_{\Sigma_b(\xi)}
\frac{\ln(1-|r(s)|^2)}{s-\mu}\dd s\right\}.
\end{equation}

Define $M^{(1)}(y,t,\mu)\coloneqq M^R(y,t,\mu)\delta^{-\sigma_3}(\mu,\xi)$. Then $M^{(1)}$ can be characterized as the solution of the RH problem including the standard normalization condition $M^{(1)}(\mu)\to I$ as $\mu\to\infty$ and the jump condition
\begin{equation}\label{M1-jump}
M^{(1)}_+(y,t,\mu)=M^{(1)}_-(y,t,\mu)J^{(1)}(y,t,\mu),\quad \mu\in\D{R},
\end{equation}
where the jump matrix is factorized as
\begin{subequations}\label{J1}
\begin{alignat}{3}\label{J1-a}
J^{(1)}(y,t,\mu)&=\begin{pmatrix}
1&  r(\mu)\delta^2(\mu,\xi)\eul^{-2\ii t\theta}\\
0& 1\\
\end{pmatrix}
\begin{pmatrix}
1& 0\\
-r^*(\mu)\delta^{-2}(\mu,\xi)\eul^{2\ii t\theta}& 1\\
\end{pmatrix}, \quad &\mu\in\Sigma_a(\xi)\\
\label{J1-b}
J^{(1)}(y,t,\mu)&=
\begin{pmatrix}
1& 0\\
-\frac{r^*(\mu)}{1-|r(\mu)|^2}\delta^{-2}_-(\mu,\xi)\eul^{2\ii t\theta}& 1\\
\end{pmatrix}
\begin{pmatrix}
1& \frac{r(\mu)}{1-|r(\mu)|^2}\delta^{2}_+(\mu,\xi)\eul^{-2\ii t\theta}\\
0& 1\\
\end{pmatrix},\quad  & \mu\in\Sigma_b(\xi).
\end{alignat}
\end{subequations}

Now let us discuss the structure of $\Sigma_a(\xi)$ and $\Sigma_b(\xi)$. First, we notice that $\hat\theta(\xi,k)$ is exactly the same as in the case of the CH equation \cite{BS08-2}. Taking into account the relation between $\mu$ and $k$ (see \eqref{xi-k-theta}), the ``signature table'' for the CH equation near the real axis leads to that for the mCH equation (the latter being, additionally, symmetric w.r.t.\ $\mu\mapsto 1/\mu$) while the ranges of values of $\xi$ for which the ``signature table'' keeps the same structure are the same. Namely, one can distinguish four ranges of values of $\xi$ for which $\Sigma_a(\xi)$ and $\Sigma_b(\xi)$ have qualitatively different structures (which, consequently, implies four qualitatively different types of large-time asymptotics):
\begin{enumerate}[(I)]
\item
$\xi>2$,
\item
$0<\xi<2$,
\item
$-\frac{1}{4}<\xi<0$,
\item
$\xi<-\frac{1}{4}$.
\end{enumerate}
Each range of values of $\xi$ is characterized by the structure of $\Sigma_a(\xi)$ (or $\Sigma_b(\xi)$): $\Sigma_a(\xi)$ is the union of disjoint intervals whose end points are the (real) stationary points of $\theta(\mu,\xi)$, i.e., the points $\mu\in\D{R}$ where $\frac{\dd\theta}{\dd \mu}(\mu,\xi)=0$, and similarly for $\Sigma_b(\xi)$. More precisely,
\begin{equation}\label{Sigma-b}
\Sigma_b(\xi)=\begin{cases}
\emptyset, & \xi>2 \\
(-\mu_0,-\frac{1}{\mu_0})\cup (\frac{1}{\mu_0},\mu_0), & 0<\xi<2 \\
(-\infty, -\mu_1)\cup (-\mu_0,-\frac{1}{\mu_0})\cup(-\frac{1}{\mu_1}, \frac{1}{\mu_1})
\cup (\frac{1}{\mu_0},\mu_0)\cup(\mu_1,+\infty), & -\frac{1}{4}<\xi<0 \\
(-\infty,+\infty), & \xi<-\frac{1}{4}.
\end{cases}
\end{equation}
Here the values of $\mu_0(\xi)>1$ and $\mu_1(\xi)>1$ are those associated (via $\kappa_j=\frac{1}{4}(\mu_j-\frac{1}{\mu_j})$, $j=0,1$) with the (real) stationary points $\kappa_0(\xi)$ and $\kappa_1(\xi)$ of $\hat\theta(k)$, i.e., the end points in the case of the CH equation. They are determined by $\xi=\frac{2-8\kappa^2}{(1+4\kappa^2)^2}$, see \cite{BS08-2}:
\[
\kappa_0^2(\xi)=\frac{\sqrt{1+4\xi}-1-\xi}{4\xi},\qquad
\kappa_1^2(\xi)=-\frac{\sqrt{1+4\xi}+1+\xi}{4\xi}
\]
($\kappa_0(\xi)$ is relevant for ranges II and III whereas $\kappa_1(\xi)$ is relevant for range III only). In analogy with the case of the CH equation, for $\xi$ in ranges I and IV, the solution $M^{(2)}$ of the RH problem (see below) decays rapidly (as $t\to+\infty$) to the identity matrix, which corresponds (in the case without discrete spectrum) to rapid decay of the resulting $\hat u(y,t)$. On the other hand, ranges II and III are those where the large-time asymptotics in the case of the CH equation are of Zakharov--Manakov type (trigonometric oscillations decaying as $t^{-1/2}$), see \cites{BKST09,BS08-2}. Our main goal in the present paper is the derivation of analogous asymptotic formulas, for ranges II and III, in the case of the mCH equation.

The next step in the transformation of the RH problem is the ``absorption'' of the triangular factors in \eqref{J1-a} and \eqref{J1-b} into the solution of a deformed RH problem, with an enhanced jump contour (having parts outside $\D{R}$). This absorption requires the triangular factors in \eqref{J1-a} and \eqref{J1-b} to have analytic continuation at least into a band surrounding $\D{R}$. With this respect we notice that, as in the case of other integrable equations (in particular, the CH equation), the reflection coefficient $r(k)$ is defined, in general, for $k\in\D{R}$ only. However, one can approximate $r(k)$ and $\frac{r(k)}{1-|r(k)|^2}$ by some rational functions with well-controlled errors (see, e.g., \cite{L15}). Alternatively, if we \emph{assume} that the initial data $\tilde u(x,0)$ decays exponentially to $0$ as $x\to\pm\infty$ (or that $\tilde u(x,0)$ has finite support in $\D{R}$), then $r(k)$ turns out to be analytic in a band containing $k\in\D{R}$ (or analytic in the whole plane) and thus there is no need to use rational approximations in order to be able to perform this absorption (see the transformation $M^{(1)}\mapsto M^{(2)}$ below). Henceforth, in order to avoid technicalities and to keep the presentation of our main result as simple as possible, we \emph{assume} that $r(k)$ (and thus $1-(r(k)r^*(k))^2$) is analytic in a domain of the complex plane containing the contours of the successive RH problems (and refer to \cite{L15} for details related to the rational approximations).

For $0<\xi<2$ and for $-\frac{1}{4}<\xi<0$, we define a contour $\Sigma\equiv\Sigma(\xi)$ consistent with the signature table for $\theta(\mu,\xi)$, see Figures \ref{f1} and \ref{f2}, respectively.

\begin{figure}[ht]
\centering\includegraphics[scale=.6]{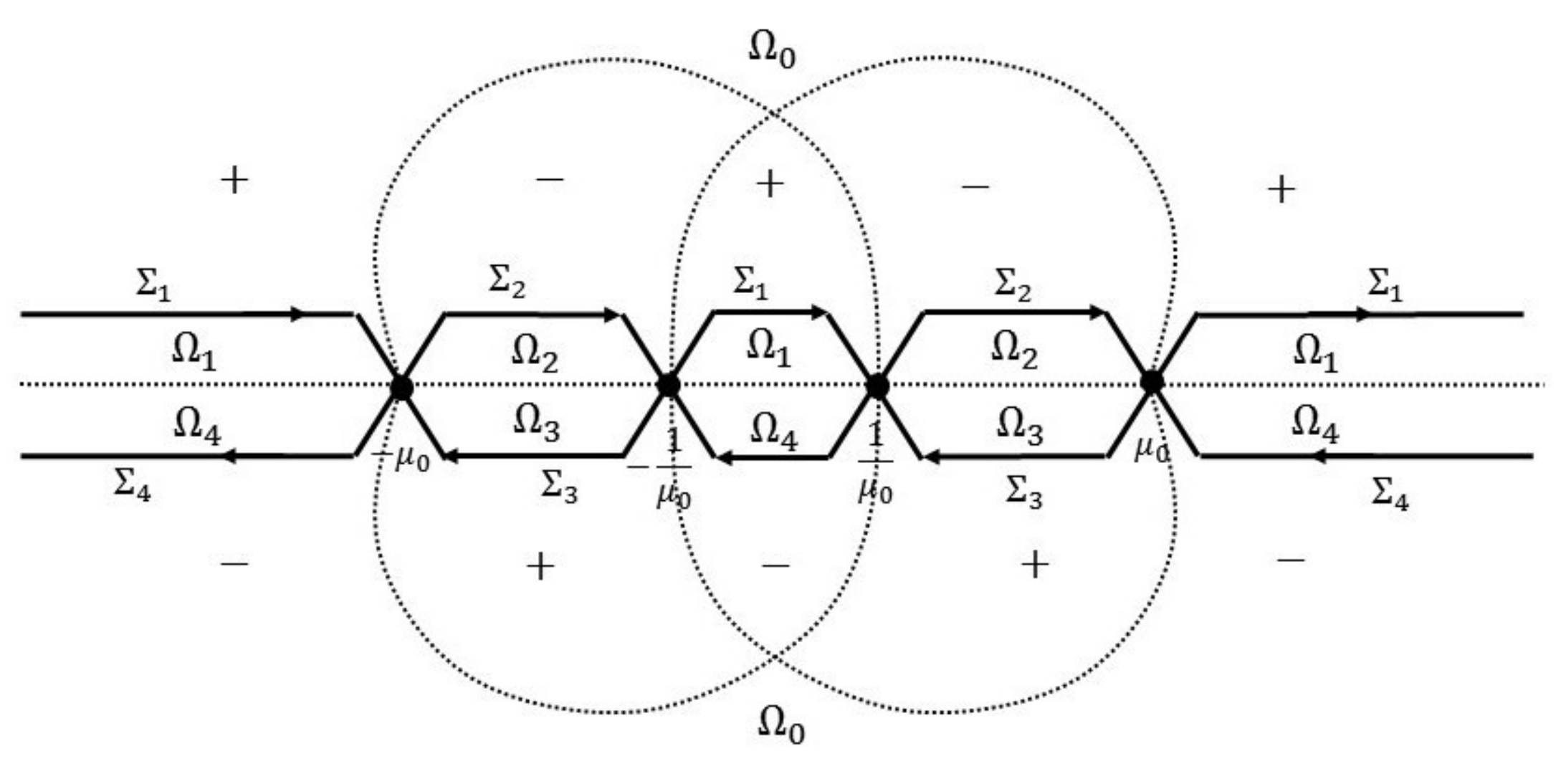}
\caption{Signature table (dotted lines), contour $\Sigma(\xi)=\cup_{j=1}^4\Sigma_j$ (solid lines) and domains $\Omega_j(\xi)$ for $0<\xi<2$.}
\label{f1}
\end{figure}

\begin{figure}[ht]
\centering\includegraphics[scale=.5]{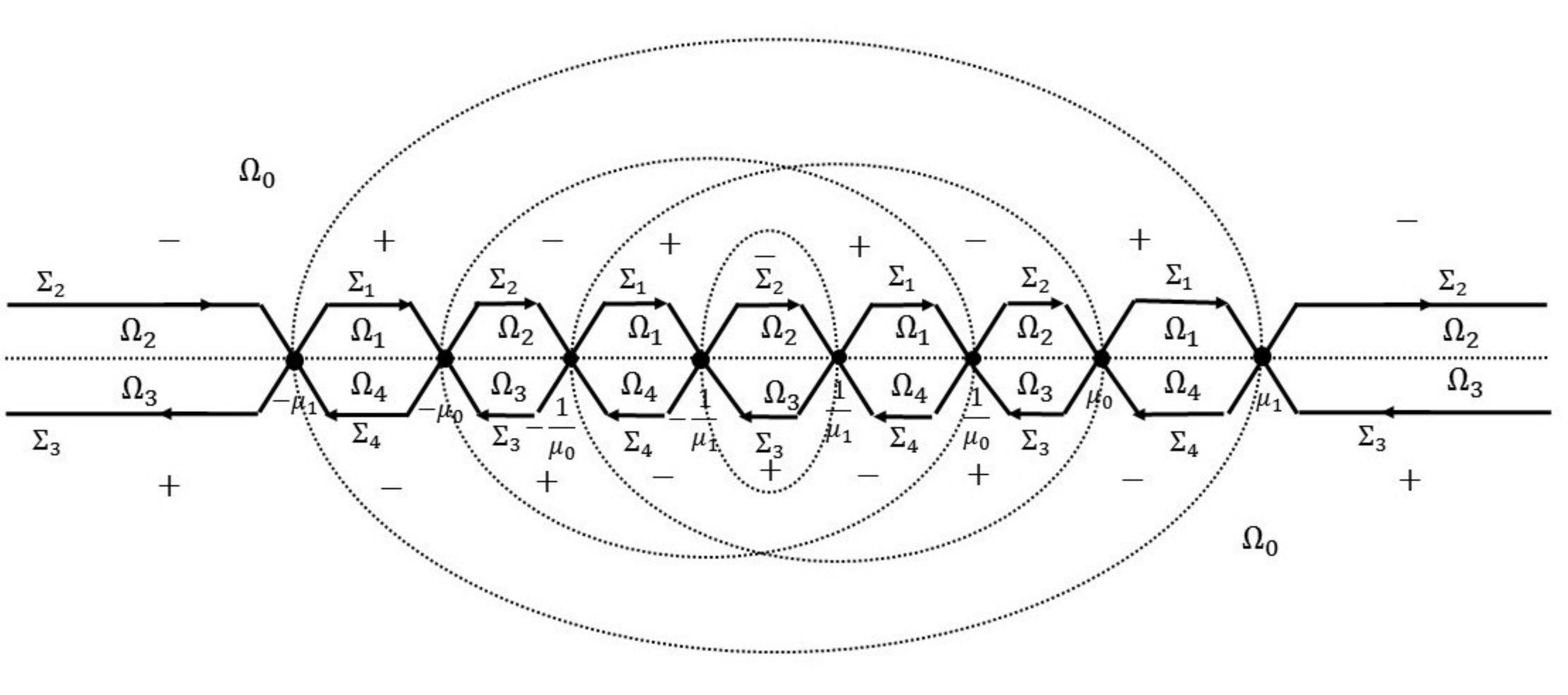}
\caption{Signature table (dotted lines), contour $\Sigma(\xi)=\cup_{j=1}^4\Sigma_j$ (solid lines) and domains $\Omega_j(\xi)$ for $-\frac{1}{4}<\xi<0$.}
\label{f2}
\end{figure}

Further, define $M^{(2)}$ by $M^{(2)}(y,t,\mu)\coloneqq M^{(1)}(y,t,\mu)P(y,t,\mu)$, where
\begin{equation}\label{M2}
P(y,t,\mu)=
\begin{cases}
I, & \mu\in\Omega_0,\\
\begin{pmatrix}
1& 0\\
r^*(\mu)\delta^{-2}(\mu,\xi)\eul^{2\ii t\theta}& 1\\
\end{pmatrix}, & \mu\in\Omega_1, \\
\begin{pmatrix}
1& -\frac{r(\mu)}{1-|r(\mu)|^2}\delta^{2}(\mu,\xi)\eul^{-2\ii t\theta}\\
0& 1\\
\end{pmatrix}, &
\mu\in\Omega_2,
\\
\begin{pmatrix}
1& 0\\
-\frac{r^*(\mu)}{1-|r(\mu)|^2}\delta^{-2}(\mu,\xi)\eul^{2\ii t\theta}& 1\\
\end{pmatrix}, & \mu\in\Omega_3,
\\
\begin{pmatrix}
1&  r(\mu)\delta^2(\mu,\xi)\eul^{-2\ii t\theta}\\
0& 1\\
\end{pmatrix}, & \mu\in\Omega_4.
\end{cases}
\end{equation}
Then $M^{(2)}(y,t,\mu)$ can be characterized as the solution of the RH problem with the standard normalization condition $M^{(2)}(\mu)\to I$ as $\mu\to\infty$ and the jump condition
\begin{equation}\label{M2-jump}
M^{(2)}_+(y,t,\mu)=M^{(2)}_-(y,t,\mu)J^{(2)}(y,t,\mu),\quad \mu\in\Sigma =
\cup_{j=1}^4 \Sigma_j,
\end{equation}
where $\Sigma_j=\overline{\Omega_0}\cap \overline{\Omega_j} $
and
\begin{equation}\label{J2}
J^{(2)}(y,t,\mu)=\begin{cases}
\begin{pmatrix}
1& 0\\
-r^*(\mu)\delta^{-2}(\mu,\xi)\eul^{2\ii t\theta}& 1\\
\end{pmatrix}, & \mu\in\Sigma_1, \\
\begin{pmatrix}
1& \frac{r(\mu)}{1-|r(\mu)|^2}\delta^{2}(\mu,\xi)\eul^{-2\ii t\theta}\\
0& 1\\
\end{pmatrix}, &
\mu\in\Sigma_2,
\\
\begin{pmatrix}
1& 0\\
\frac{r^*(\mu)}{1-|r(\mu)|^2}\delta^{-2}(\mu,\xi)\eul^{2\ii t\theta}& 1\\
\end{pmatrix}, & \mu\in\Sigma_3,
\\
\begin{pmatrix}
1&  -r(\mu)\delta^2(\mu,\xi)\eul^{-2\ii t\theta}\\
0& 1\\
\end{pmatrix}, & \mu\in\Sigma_4.
\end{cases}
\end{equation}

The RH problem for $M^{(2)}$ is such that uniform decay (as $t\to+\infty$) of the jump matrix is violated only near the stationary phase points of $\theta(\mu)$. The large-time analysis, with appropriate estimates, of such problems involves the ``comparison'' of the RH problem with that modified in small vicinities of the stationary phase points, using rescaled spectral parameters as well as approximations of the jump matrices in these vicinities \cites{DZ93}.

In our large-time analysis for $M^{(2)}$, we follow the strategy presented in \cite{L15}.

\begin{stepone*}
Add to $\Sigma$ small circles $\gamma_j$ surrounding $\mu_j$, $j=0,1$ and their images $-\gamma_{j}$ and $\pm \gamma_{j}^{-1}$ under the mappings $\mu\mapsto -\mu$ (surrounding $-\mu_j$) and $\mu\mapsto 1/\mu$ (surrounding $\pm 1/\mu_j$) respectively.
\end{stepone*}
\begin{steptwo*}
Inside the circles around $\mu_j$, $j=0,1$, define (explicitly)   $m_0(y,t,\mu)$ as functions that exactly satisfy the jump conditions with jumps obtained from $J^{(2)}$ by replacing $r(\mu)$ with $r(\mu_0)$ and $r(\mu_1)$, respectively, and by replacing $\delta^2(\mu,\xi)\eul^{-2\ii t\theta(\mu,\xi)}$ with its large-time approximations.
\end{steptwo*}
\begin{stepthree*}
Define $m_0(y,t,\mu)$ inside the other small contours using the symmetries $m_0(\mu)=\overline{m_0(1/\bar\mu)}$ and $m_0(\mu)=\sigma_3 \overline{m_0(-\bar \mu)}\sigma_3 $ (which are consistent with the symmetries of $M^{(2)}(\mu)$).
\end{stepthree*}
\begin{stepfour*}
Define $\hat m(\mu)$ by
\[
\hat m(y,t,\mu)=\begin{cases}
M^{(2)}(y,t,\mu)m_0^{-1}(y,t,\mu),&\text{inside } \pm\gamma_{j}
\text{ and }\pm \gamma_{j}^{-1},\\
M^{(2)}(y,t,\mu), & \text{otherwise},
\end{cases}
\]
Then $\hat m(\mu)$ satisfies the conditions of  the RH problem
\[
\begin{cases}
\hat m_+(y,t,\mu)=\hat m_-(y,t,\mu)\hat J(y,t,\mu),& \mu\in\hat\Sigma\coloneqq \Sigma\cup_j \{\pm\gamma_{j}\}\cup_j\{\pm \gamma_{j}^{-1}\}, \\
\hat m(y,t,\mu)\to I,& \mu\to\infty,
\end{cases}
\]
where
\[
\hat J(y,t,\mu)=\begin{cases}
m_0^{-1}(y,t,\mu), & \mu\in \cup_j \{\pm\gamma_{j}\}\cup_j\{\pm \gamma_{j}^{-1}\},\\
m_{0-}^{-1}(y,t,\mu)J^{(2)}(y,t,\mu)m_{0+}(y,t,\mu), &
k\in \Sigma\cap\{k\mid k\text{ inside }\cup_j\{\pm\gamma_{j}\}\cup_j\{\pm \gamma_{j}^{-1}\}\},\\
J^{(2)}(y,t,\mu),& \text{otherwise}.
\end{cases}
\]
On the other hand, the unique solution of this problem can be expressed in terms of the solution $\Theta(\mu)$ of the singular integral equation (see \cite{L15}*{Lemma 2.9}):
\begin{equation}\label{hat-m}
\hat m(y,t,\mu)= I + \frac{1}{2\pi\ii}\int_{\hat\Sigma}\Theta(y,t,s)\hat w(y,t,s)\frac{\dd s}{s-\mu}.
\end{equation}
Here $\hat w(y,t,s)\coloneqq\hat J(y,t,s)-I$ and $\Theta\in I+L^2(\hat \Sigma)$ is the solution of the integral equation
\[
\Theta(\mu) -  \C{C}_{\hat w} \Theta(\mu)=I,
\]
where $\C{C}_{\hat{w}}\colon L^2(\hat\Sigma)+L^{\infty}(\hat\Sigma)\to L^2(\hat\Sigma)$ is an integral operator defined with the help of the singular Cauchy operator: $\C{C}_{\hat{w}}f\coloneqq\C{C}_-(f\hat{w})$, where $\C{C}_-=\frac{1}{2}(-I+S_{\hat\Sigma})$ and $S_{\hat\Sigma}$ is the operator associated with $\hat\Sigma$ and defined by the principal value
of the Cauchy integral:
\[
(S_{\hat\Sigma}f)(\mu)=\frac{1}{2\pi\ii}\int_{\hat \Sigma}\frac{f(s)}{s-\mu} ds,\quad \mu\in\hat\Sigma.
\]
\end{stepfour*}
\begin{stepfive*}
Estimate the large-time behavior of $\hat m(y,t,\mu)$ at $\mu=\ii$ and $\mu=0$ taking into account the following facts:
\begin{enumerate}[\textbullet]
\item
The main contribution to the r.h.s.\ of \eqref{hat-m} comes from the integrals over the small contours, where $\hat w(y,t,\mu)=m_0^{-1}(y,t,\mu)-I$:
\begin{equation}\label{hat-m-as}
\hat m(y,t,\mu)=I+\frac{1}{2\pi\ii}\int_{\cup_j\{\pm\gamma_{j}\}\cup_j\{\pm\gamma_{j}^{-1}\}}\frac{m_0^{-1}(y,t,s)-I}{s-\mu}\dd s+ \osmall(t^{-1/2}).
\end{equation}
Henceforth the error estimates are uniform for $\varepsilon<\xi<2-\varepsilon$ and $-\frac{1}{4}+\varepsilon<\xi<-\varepsilon$, for any small $\varepsilon>0$. For detailed  estimates, see \cite{L15}.
\item
In turn, the main contribution to $m_0^{-1}(y,t,\mu)-I$ comes from the asymptotics of the RH problem for parabolic cylinder functions (involved in the construction of $m_0(y,t,\mu)$), see \cite{L15}*{Appendix B}, which can be given explicitly.
\end{enumerate}
\end{stepfive*}

\subsection{Range $0<\xi<2$}

This range is characterized by the presence of four real critical points: $\pm\mu_0$ and $\pm\mu_0^{-1}$.
\subsubsection{Construction of $m_0$}

First, we approximate $\ii t\theta(\mu,\xi)$ using \eqref{xi-k-theta},
the relation
\begin{equation}\label{kappa-mu}
\kappa_0=\frac{1}{4}\left(\mu_0-\frac{1}{\mu_0}\right)
\end{equation}
between $\mu_0$ and $\kappa_0$, and the approximation for $\hat\theta(k,\xi)$ near $\kappa_0$, see \cite{BS08-2}:
\[
\hat\theta(k,\xi)\approx \hat\theta(\kappa_0) + 8f_0(\kappa_0)(k-\kappa_0)^2,
\]
where
\begin{equation}\label{f0-hat-theta}
f_0(\kappa_0)=\frac{\kappa_0(3-4\kappa_0^2)}{(1+4\kappa_0^2)^3}, \qquad
\hat\theta(\kappa_0)=-\frac{16\kappa_0^3}{(1+4\kappa_0^2)^2}.
\end{equation}
We have $-\ii t\theta(\mu,\xi)\approx -\ii t \hat\theta(\kappa_0)-\frac{\ii\hat\mu^2}{4}$,
where the scaled spectral variable $\hat\mu$ is introduced by
\begin{equation}\label{hat-mu}
\mu-\mu_0=\frac{\hat\mu}{(1+\mu_0^{-2})\sqrt{2f_0 t}}.
\end{equation}

Now we approximate $\delta(\mu,\xi)$ near $\mu=\mu_0$. From \eqref{delta} we have
\begin{align}\label{delta-mu}
\delta(\mu,\xi)
&=\exp\left\{\frac{1}{2\pi\ii}\left(\int_{-\mu_0}^{-1/\mu_0}+ \int^{\mu_0}_{1/\mu_0}\right)\frac{\ln(1-|r(s)|^2)}{s-\mu}\dd s\right\}\notag\\
&=\left(\frac{\mu-\mu_0}{\mu-1/\mu_0}\right)^{\ii h_0} \left(\frac{\mu+1/\mu_0}{\mu+\mu_0}\right)^{\ii h_0}\eul^{\chi(\mu)},
\end{align}
where 
\begin{align*}
h_0&=-\frac{1}{2\pi}\ln(1-|r(\mu_0)|^2),\\
\chi(\mu)&=\frac{1}{2\pi\ii}\left(\int_{-\mu_0}^{-1/\mu_0}
+ \int^{\mu_0}_{1/\mu_0}
\right) \ln\frac{1-|r(s)|^2}{1-|r(\mu_0)|^2}\frac{\dd s}{s-\mu}
\end{align*}
(notice that $|r(\mu)|=|r(-\mu)|=|r(1/\mu)|$). Therefore (cf. \cite{BS08-2}),
\[
\delta(\mu,\xi)\approx (\mu-\mu_0)^{\ii h_0}\left(\frac{\mu_0+1/\mu_0}{2\mu_0(\mu_0-1/\mu_0)}\right)^{\ii h_0}\eul^{\chi(\mu_0)}
= \hat\mu^{\ii h_0}(128 f_0\kappa_0^2 t)^{-\frac{\ii h_0}{2}}\eul^{\chi(\mu_0)}
\]
and thus
\begin{equation}\label{d-e-appr}
\delta(\mu,\xi)\eul^{-\ii t \theta(\mu,\xi)} \approx \delta_{\mu_0}(\xi,t)\hat\mu^{\ii h_0}
\eul^{-\frac{\ii\hat\mu^2}{4}},
\end{equation}
where
\begin{equation}\label{d-mu}
\delta_{\mu_0}(\xi,t)=
\eul^{-\ii t \hat\theta(\kappa_0(\mu_0))}
\eul^{\chi(\mu_0)}
(128 f_0(\kappa_0(\mu_0))\kappa_0^2(\mu_0) t)^{-\frac{\ii h_0}{2}}.
\end{equation}

The approximation \eqref{d-e-appr} suggests introducing $m_0(y,t,\mu)$ (near $\mu=\mu_0$) as follows:
\begin{equation}\label{m0}
m_0(y,t,\mu)=D(\xi,t)m^X(\xi,\hat\mu)D^{-1}(\xi,t),
\end{equation}
where $D(\xi,t)=\delta_{\mu_0}^{\sigma_3}(t)$
and $m^X(\xi,\hat\mu)$ is the solution of the RH problem, in the $\hat\mu$-complex plane,
whose solution is given in terms of parabolic cylinder functions
\cite{L15} (with $q=-\bar r(\mu_0)$).

Since (see \eqref{hat-mu}) finite values of $\mu$ correspond to growing (with $t$) values of $\hat\mu$, the large-time asymptotics of $m_0(y,t,\mu)$ for $\mu$ on the small contours surrounding $\pm\mu_0$ and $\pm\frac{1}{\mu_0}$ involves the large-$\hat\mu$ asymptotics of $m^X(\xi,\hat\mu)$, which is given by (see \cite{L15}*{Appendix B})
\begin{equation}\label{mX}
m^X(\xi,\hat\mu)=I+\frac{\ii}{\hat\mu}\begin{pmatrix}
0 & -\beta_{\mu_0}(\xi) \\ \bar \beta_{\mu_0}(\xi) & 0
\end{pmatrix} +\ord(\hat\mu^{-2})
\end{equation}
with
\begin{equation}\label{beta}
\beta_{\mu_0}(\xi)=\sqrt{h_0}\eul^{\ii\left(\frac{\pi}{4}-\arg (-\bar r(\mu_0))+\arg\Gamma(\ii h_0)\right)},
\end{equation}
where $\Gamma$ is Euler's gamma function.
From \eqref{hat-mu}, \eqref{m0} and \eqref{mX} we have
\begin{align}\label{m0-as}
m_0^{-1}(y,t,\mu) 
&=D(\xi,t)(m^X)^{-1}(\xi, \hat\mu(\mu))D^{-1}(\xi,t)\notag\\
&=D(\xi,t)\left(I-\frac{\ii}{\hat\mu(\mu)}
	\begin{pmatrix}
	0 & -\beta_{\mu_0}(\xi) \\ \bar \beta_{\mu_0}(\xi) & 0
	\end{pmatrix}\right)
	D ^{-1}(\xi,t) +\ord(t^{-1})\notag \\
& =I + \frac{B(\xi,t)}{\sqrt{t}(\mu-\mu_0)} + \ord(t^{-1}),
\end{align}
where
\begin{equation}\label{B}
B(\xi,t)=\begin{pmatrix}
0 & B_0(\xi,t) \\ \bar  B_0(\xi,t) & 0
\end{pmatrix}
\quad\text{with}\ \
B_0(\xi,t)=\frac{\ii \delta_{\mu_0}^2(\xi,t)\beta_{\mu_0}(\xi)}{(1+\mu_0^{-2})\sqrt{2 f_0(\kappa_0(\mu_0))}}.
\end{equation}
Here the estimate $\ord(t^{-1})$ is uniform for $\xi$ and $\mu$ such that $\varepsilon_1<\xi<2-\varepsilon_1$ and $|\mu-\mu_0|=\varepsilon_2$ for any small positive $\varepsilon_j$, $j=1,2$.

\subsubsection{Asymptotics for $\hat m$}\label{s:as-hat-m}

In view of our algorithm for representing $u$ in terms of the solution of the associated regular RH problem, see \eqref{M-all-1}, \eqref{hat-M-i}, \eqref{u-and-x}, and \eqref{utilde}, we need to know the asymptotics for $\hat m(y,t,0)$, $\hat m(y,t,\ii)$, and $\hat m_1(y,t)$, where $\hat m_1$ is extracted from the expansion $\hat m(y,t,\mu)=\hat m(y,t,\ii) + \hat m_1(y,t)(\mu-\mu_0) + \ord((\mu-\mu_0)^2)$ as $\mu\to\mu_0$. By \eqref{m0-as} and the residue theorem, the leading contributions of the integral over $\gamma_0$ into \eqref{hat-m-as} for these quantities are, respectively,
\begin{equation}\label{contrib-mu0}
\frac{B}{\mu_0\sqrt{t}}, \quad \frac{B}{(\mu_0-\ii)\sqrt{t}} \ \ \text{and}\
\frac{B}{(\mu_0-\ii)^2\sqrt{t}}.
\end{equation}

In order to take into account the contributions of all small contours, we extend the definition of $m_0$ by symmetries (as indicated in Step (iii)). This gives
\begin{align}\label{mu0-as}
\hat m(y,t,0) 
&=I +\left(\frac{B}{\mu_0}- \frac{\bar B}{\mu_0} -\frac{1}{\mu_0^2}\frac{\bar B}{\mu_0^{-1}}+ \frac{1}{\mu_0^2}\frac{B}{\mu_0^{-1}}\right)\frac{1}{\sqrt{t}}+\osmall(t^{-1/2})\notag\\
&=I+\frac{4i\Im {B_0}(\xi,t)}{\mu_0\sqrt{t}}\begin{pmatrix}0&1\\-1& 0\end{pmatrix}+\osmall(t^{-1/2}),
\end{align}
\begin{align}\label{mu-i-as}
\hat m(y,t,\ii) 
&=I+\left(\frac{B}{\mu_0-\ii}+ \frac{\bar B}{-\mu_0-\ii} -\frac{1}{\mu_0^2}\frac{\bar B}{\mu_0^{-1}-\ii}
- \frac{1}{\mu_0^2}\frac{B}{-\mu_0^{-1}-\ii}
\right)\frac{1}{\sqrt{t}}+\osmall(t^{-1/2})\notag\\
&=I+\frac{2\ii\Im {B_0}(\xi,t)}{\mu_0\sqrt{t}} \begin{pmatrix}
0 & 1 \\ -1 & 0
\end{pmatrix} + \osmall(t^{-1/2}),
\end{align}
and
\begin{align}\label{mu1-as}
\hat m_1(y,t) &=\left(
\frac{B}{(\mu_0-\ii)^2}+ \frac{\bar B}{(-\mu_0-\ii)^2}
-\frac{1}{\mu_0^2}\frac{\bar B}{(\mu_0^{-1}-\ii)^2}
- \frac{1}{\mu_0^2}\frac{B}{(-\mu_0^{-1}-\ii)^2}
\right)\frac{1}{\sqrt{t}}+\osmall(t^{-1/2})\notag\\
&={\frac{4}{\sqrt{t}} \begin{pmatrix}
0 & \Re\frac{B_0}{(\mu_0-\ii)^2} \\ \Re\frac{\bar B_0}{(\mu_0-\ii)^2} & 0\end{pmatrix}} + \osmall(t^{-1/2}).
\end{align}

\subsubsection{From $\hat m$ back to $M^{R}$}

In Section \ref{s:as-hat-m} we presented the large-time asymptotics of $\hat m(y,t,\mu)$ (and thus of $M^{(2)}(y,t,\mu)$) for the dedicated values of $\mu$. Since $P(y,t,0)=0$ whereas $P(y,t,\mu)$ tends to $I$ exponentially fast, as $t\to+\infty$ for all $\mu$ close to $\ii$, in order to obtain the leading terms of the asymptotics for $M^{R}(y,t,\mu) =M^{(1)}(y,t,\mu)\delta^{\sigma_3}(\mu,\xi)=M^{(2)}(y,t,\mu)P^{-1}(y,t,\mu)\delta^{\sigma_3}(\mu, \xi)$, we need to know $\delta(\mu,\xi)$ \eqref{delta} for $\mu=0$ and $\mu$ near $\ii$.

Due to the symmetry $|r(\mu)|=|r(-\mu)|$ we have
\begin{equation}\label{delta-0}
\delta(0, \xi)=\exp\left\{\frac{1}{2\pi\ii}\int_{\Sigma_b(\xi)}
\frac{\ln(1-|r(s)|^2)}{s}\dd s\right\} \equiv 1.
\end{equation}
As $\mu\to \ii$, denote $\delta(\mu, \xi)=\eul^{I_0 + I_1(\mu-\ii)+ \dots}$. Then (using again the symmetry $|r(\mu)|=|r(-\mu)|$)
\[
I_0=\frac{1}{2\pi\ii}\int_{\Sigma_b(\xi)}
\frac{\ln(1-|r(s)|^2)}{s-\ii}\dd s=\frac{1}{\pi}\int_{1/\mu_0}^{\mu_0}
\frac{\ln(1-|r(s)|^2)}{s^2+1}\dd s.
\]
On the other hand,
\begin{align*}
I_1&=\frac{1}{2\pi\ii}\int_{1/\mu_0}^{\mu_0}\ln(1-|r(s)|^2)\left(\frac{1}{(s-\ii)^2}+\frac{1}{(-s-\ii)^2}\right)\dd s\\ 
&=\frac{1}{\pi\ii}\int_{1/\mu_0}^{\mu_0}
\ln(1-|r(s)|^2)\frac{s^2-1}{(s^2+1)^2}\dd s \equiv 0,
\end{align*}
the latter equality being due to the symmetry $|r(\mu)|=|r(\mu^{-1})|$.
Therefore,
\begin{equation}\label{delta-i}
\delta(\mu, \xi)= \delta(\ii, \xi)+ \ord((\mu-\ii)^2) \ \  \text{with}\
\delta(\ii, \xi)=
\exp\left\{\frac{1}{\pi}\int_{1/\mu_0}^{\mu_0}
\frac{\ln(1-|r(s)|^2)}{s^2+1}\dd s
\right\}.
\end{equation}
Therefore, we have the following asymptotics for $M^{R}(y,t,0)$,  $M^{R}(y,t,\ii)$, and $M^{R}_1(y,t)$, where $M^{R}(y,t,\mu)=M^{R}(y,t,\ii)+M^{R}_1(y,t)(\mu-\ii)+\ord((\mu-\ii)^2)$:
\begin{equation}\label{M-R-as}
\begin{split}
M^{R}(y,t,0) 
& =\hat m(y,t,0)=I+\frac{4\ii\Im B_0(\xi,t)}{\mu_0\sqrt{t}} \begin{pmatrix}
	0 & 1 \\ -1 & 0
	\end{pmatrix} + \osmall(t^{-1/2}), \\
M^{R}(y,t,\ii) 
&=\hat m(y,t,\ii)\delta^{\sigma_3}(\ii,\xi)+\ord(\eul^{-\varepsilon t})\\
&= \left(I+\frac{2\ii\Im B_0(\xi,t)}{\mu_0\sqrt{t}} \begin{pmatrix}
	0 & 1 \\ -1 & 0
\end{pmatrix}\right)\delta^{\sigma_3}(\ii,\xi)+ \osmall(t^{-1/2}), \\
M^{R}_1(y,t) 
&=\hat m_1(y,t)\delta^{\sigma_3}(\ii,\xi)+\ord(\eul^{-\varepsilon t})\\ &=\frac{4}{\sqrt{t}} \begin{pmatrix}
	0 & \Re\frac{B_0}{(\mu_0-\ii)^2} \\ \Re\frac{\bar B_0}{(\mu_0-\ii)^2} & 0
\end{pmatrix}\delta^{\sigma_3}(\ii,\xi) + \osmall(t^{-1/2}),
\end{split}
\end{equation}
where $B_0(\xi,t)$ is given by \eqref{B} and $\delta(\ii,\xi)$
is given by \eqref{delta-i}.

\subsubsection{Large-time asymptotics of $u$}

Combining the asymptotics for $M^R(\mu)$ \eqref{M-R-as} with \eqref{hat-M-i}, \eqref{u-and-x}, \eqref{M-to-tilde}, and \eqref{to-M-reg}, we can obtain the leading term of the large-time asymptotics of $u(x,t)$.

Introducing $\eta:=\frac{2\Im B_0}{\mu_0\sqrt{t}}$,
from  \eqref{M-R-as}  we have:
\begin{equation}\label{Delta-as}
\Delta(y,t)=\sigma_1 [M^R(y,t,0)]^{-1} =
\begin{pmatrix}
2\ii \eta & 1 \\ 1 & -2\ii\eta
\end{pmatrix} + \osmall(t^{-1/2}).
\end{equation}
Therefore, for
\begin{equation}\label{M-all}
M(\mu)=\left(I-\frac{1}{\mu}\sigma_1\right)^{-1} \left(I-\frac{1}{\mu}\Delta\right)
M^R(\mu)
\end{equation}
we have $M(\mu)=I_1(\mu) I_2(\mu) M^R(\mu) +\osmall(t^{-1/2})$,
where
\begin{subequations}\label{I123}
\begin{align}
\label{I1}
I_1(\mu) &=\begin{pmatrix}
\frac{\mu^2}{\mu^2-1} & \frac{\mu}{\mu^2-1} \\
\frac{\mu}{\mu^2-1} & \frac{\mu^2}{\mu^2-1}
\end{pmatrix}=\begin{pmatrix}
\frac{1}{2} & -\frac{\ii}{2} \\
-\frac{\ii}{2} & \frac{1}{2}
\end{pmatrix} -\frac{\ii}{2} I (\mu-\ii) +\ord((\mu-\ii)^2), \\
\label{I2}
I_2(\mu) &=\begin{pmatrix}
1-\frac{2\ii \eta}{\mu} & -\frac{1}{\mu} \\
-\frac{1}{\mu} & 1+\frac{2\ii \eta}{\mu}
\end{pmatrix}\notag\\
&=\begin{pmatrix}
1-2\eta & \ii \\
\ii & 1+2\eta
\end{pmatrix} + \begin{pmatrix}
-2\ii \eta  & -1 \\
-1 & 2\ii \eta
\end{pmatrix} (\mu-\ii) +\ord((\mu-\ii)^2),\\
\label{I3}
M^R(\mu) &= \begin{pmatrix}
1 & \ii \eta \\
-\ii \eta  & 1
\end{pmatrix}\delta^{\sigma_3}(\ii) + \begin{pmatrix}
0 & \beta_1 \\
\beta_2 & 0
\end{pmatrix}\delta^{\sigma_3}(\ii) (\mu-\ii) +\ord((\mu-\ii)^2),
\end{align}
\end{subequations}
with
\begin{equation}\label{betas}
\beta_1=\frac{4}{\sqrt{t}}\Re\frac{B_0}{(\mu_0-\ii)^2}, \quad
\beta_2=\frac{4}{\sqrt{t}}\Re\frac{\bar B_0}{(\mu_0-\ii)^2}.
\end{equation}
Substituting \eqref{I123} into \eqref{M-all} and keeping the terms
of order $t^{-1/2}$ we have
\[
M(\mu)=\begin{pmatrix}
(1 - \eta)\delta(\ii) & 0 \\
0  & (1 + \eta)\delta^{-1}(\ii)
\end{pmatrix} +
\begin{pmatrix}
0 & (\beta_1 + \eta)\delta^{-1}(\ii) \\
(\beta_2 - \eta)\delta(\ii)  & 0
\end{pmatrix}(\mu-\ii) +\osmall((\mu-\ii)t^{-1/2})
\]
and thus (see \eqref{hat-M-i})
\[
a_1=(1 - \eta)\delta(\ii) + \osmall(t^{-1/2}),
\ \ a_2=(\beta_1 + \eta)\delta^{-1}(\ii) + \osmall(t^{-1/2}),
\ \ a_3=(\beta_2 - \eta)\delta(\ii) + \osmall(t^{-1/2}).
\]
It follows (see \eqref{u-and-x}) that
\begin{subequations}\label{u-x-as}
\begin{alignat}{3}
\label{u-y-as}
\hat u(y,t) &=-(\beta_1+\beta_2)+\osmall(t^{-1/2}) =
\frac{8(1-\mu_0^2)}{(1+\mu_0^2)^2\sqrt{t}}\Re B_0 +\osmall(t^{-1/2}),\\
\label{x-y-as}
x(y,t) &=y+2\ln((1-\eta)\delta(\ii)) +  \osmall(t^{-1/2}) =
y+y_0(\xi) +\ord(t^{-1/2}),
\end{alignat}
\end{subequations}
where (see \eqref{delta-i})
$y_0(\xi) =\frac{2}{\pi}\int_{1/\mu_0}^{\mu_0}
\frac{\ln(1-|r(s)|^2)}{s^2+1}\dd s$.

Recalling the definition \eqref{B} of $B_0$ and introducing the real-valued functions $\varphi_\delta(\xi,t)$ and $\varphi_\beta(\xi)$ (see \eqref{beta} and \eqref{d-mu}) by
\[
\beta_{\mu_0}(\xi)=\sqrt{h_0}\eul^{\ii \varphi_\beta(\xi)}, \quad
\delta_{\mu_0}^2(\xi,t)= \eul^{\ii \varphi_\delta(\xi,t)}
\]
we have $B_0=\frac{\sqrt{h_0}}{(1+\mu_0^{-2})\sqrt{2 f_0}}\eul^{\ii(\frac{\pi}{2}
	+\varphi_\delta (\xi,t))+\varphi_\beta(\xi))}$ and thus
\begin{equation}\label{B1}
	\Re B_0(\xi,t)=\frac{\sqrt{h_0}}{(1+\mu_0^{-2})\sqrt{2 f_0}}
	\cos\left\{\frac{\pi}{2}
	+\varphi_\delta (\xi,t)+ \varphi_\beta(\xi)\right \}.
\end{equation}
Substituting \eqref{B1} into \eqref{u-y-as} gives the asymptotics of the solution of the Cauchy problem for the mCH equation (in the form \eqref{mCH2}) expressed parametrically, in the $(y,t)$ variables. Recalling the definitions of $f_0$, $\varphi_\delta$, $\varphi_\beta$, $\beta_{\mu_0}$ (see \eqref{f0-hat-theta}, \eqref{d-mu}, \eqref{beta}) and the relationship \eqref{kappa-mu} between $\mu_0$ and $\kappa_0$ we obtain the following large-time asymptotics along the rays $\frac{y}{t}=\xi$ for $0<\xi<2$:
\begin{equation}\label{as-hat-u}
\hat u(y,t)=\frac{C_1(\xi)}{\sqrt{t}}\cos\left\{C_2(\xi) t + C_3(\xi)\ln t + C_4(\xi)\right\} +\osmall(t^{-1/2}),
\end{equation}
where
\begin{subequations}\label{C1-4}
\begin{alignat}{3}
\label{C1}
C_1(\xi)&=-\left(\frac{8h_0\kappa_0}{3-4\kappa_0^2}\right)^{\frac{1}{2}},\\
\label{C2}
C_2(\xi)&=\frac{32\kappa_0^3}{(1+4\kappa_0^2)^2}, \\
\label{C3}
C_3(\xi)&=-h_0,\\
\label{C4}
C_4(\xi)&= \frac{3\pi}{4}-\frac{1}{\pi}\left(\int_{-\mu_0}^{-1/\mu_0}
+ \int^{\mu_0}_{1/\mu_0}\right)\ln\frac{1-|r(s)|^2}{1-|r(\mu_0)|^2}\frac{\dd s}{s-\mu_0}
-h_0\ln \frac{128\kappa_0^3(3-4\kappa_0^2)}{(1+4\kappa_0^2)^3}\notag \\
&\quad -\arg(-\bar r(\mu_0))+\arg\Gamma(\ii h_0),
\end{alignat}
\end{subequations}
taking into account that $h_0$, $\kappa_0$, and $\mu_0$ are defined as functions of $\xi$.

In order to express the asymptotics of $\tilde u(x,t)=\hat u(y(x,t),t)$ in the $(x,t)$ variables, we notice that \eqref{x-y-as} reads
\[
\frac{y}{t}=\frac{x}{t}-\frac{y_0}{t}+\ord(t^{-3/2})
\]
and thus introducing $\zeta:=\frac{x}{t}$ gives $C_j(\xi)=C_j(\zeta)+\ord(t^{-1})$, $j=1,\dots,4$ and
\[
C_2(\xi)t=C_2(\zeta)t-\frac{\dd C_2}{\dd\zeta}(\zeta)y_0(\zeta)+\osmall(1).
\]
It follows that the leading term of the asymptotics for $\tilde u(x,t)$ can be obtained from the r.h.s.\ of \eqref{as-hat-u}, where 
\begin{enumerate}[(i)]
\item
$C_j(\xi)$ are replaced by $C_j(\zeta)$ for $j=1,2,3$, and
\item
$C_4(\xi)$ is replaced by $\tilde C_4(\zeta):=C_4(\zeta)-C_2'(\zeta)y_0(\zeta)$.
\end{enumerate}
In turn, calculating $C_2'(\zeta)$ in terms of $\kappa_0(\zeta)$
and using \eqref{C2} and
$\zeta=\frac{2-8\kappa_0^2}{(1+4\kappa_0^2)^2}$,
we get $C_2'(\zeta)=-2\kappa_0$ and thus
\begin{equation}\label{C4-1}
\tilde C_4(\zeta)=C_4(\zeta) +\frac{4\kappa_0(\zeta)}{\pi}
\int_{1/\mu_0}^{\mu_0}
\frac{\ln(1-|r(s)|^2)}{s^2+1}\dd s.
\end{equation}

The asymptotic analysis we have presented above can be summarized in the following

\begin{theorem}\label{th1}
In the solitonless case, the solution $\tilde u(x,t)$ of the Cauchy problem for the mCH equation in the form \eqref{mCH2} has the following large-time asymptotics along the rays $\frac{x}{t}\eqqcolon\zeta$ in the sector of the $(x,t)$ half-plane $0<\zeta<2$:
\begin{equation}\label{as-tilde-u}
\tilde u(x,t)=\frac{C_1(\zeta)}{\sqrt{t}}\cos\left\{C_2(\zeta) t + C_3(\zeta)\ln t+\tilde C_4(\zeta)\right\}+\osmall(t^{-1/2})
\end{equation}
with $C_1,C_2,C_3$ defined by \eqref{C1}-\eqref{C3}, and $\tilde C_4$ defined by \eqref{C4-1}-\eqref{C4}. Moreover, in these definitions $h_0=-\frac{1}{2\pi}\ln(1-|r(\mu_0)|^2)$, $\kappa_0(\zeta)=\left(\frac{\sqrt{1+4\zeta}-1-\zeta}{4\zeta}\right)^{\frac{1}{2}}$, and $\mu_0(\zeta)>1$ is characterized by the relation $\kappa_0(\zeta)=\frac{1}{4}(\mu_0(\zeta)-\mu_0(\zeta)^{-1})$.
\end{theorem}

By using the relation \eqref{utilde} between $\tilde u$ and $u$ we immediately obtain, as a corollary, the large-time asymptotics for $u(x,t)$ in the sector $1<\frac{x}{t}<3$. 

\begin{theorem}[$1^{\text{st}}$ oscillatory region]\label{th2}
In the solitonless case, the solution $u(x,t)$ of the Cauchy problem \eqref{mCH1-ic} for the mCH equation has the following large-time asymptotics in the sector of the $(x,t)$ half-plane defined by $1<\zeta\coloneqq\frac{x}{t}<3$:
\begin{equation}\label{as-u}
u(x,t)=1+\frac{C_1(\zeta-1)}{\sqrt{t}}\cos\left\{C_2(\zeta-1) t + C_3(\zeta-1)\ln t + \tilde C_4(\zeta-1)\right\}+\osmall(t^{-1/2}).
\end{equation}
The error term is uniform in any sector $1+\varepsilon<\zeta<3-\varepsilon$ where $\varepsilon$ is a small positive number.
\end{theorem}

\subsection{Range $-\frac{1}{4}<\xi<0$}

This range is characterized by the presence of eight real critical points: $\pm\mu_0$, $\pm\mu_1$, $\pm\mu_0^{-1}$, and $\pm\mu_1^{-1}$, see Figure \ref{f2}. Similarly to the range $0<\xi<2$, we proceed, first, by evaluating the contribution to \eqref{hat-m-as} from $\gamma_0$ and $-\gamma_1$ and then by using the symmetries $\mu\mapsto -\mu$ and $\mu\mapsto 1/\mu$. Notice that choosing $-\gamma_1$ surrounding $-\mu_1$ is suggested by the structure of $\Sigma_b(\xi)$ \eqref{Sigma-b}: the parts of $\Sigma_b(\xi)$ ending at $\mu_0$ and at $-\mu_1$ are located to the left of these points. This implies that the construction of the local approximation near $-\mu_1$ follows exactly the same lines as for $\mu_0$, the only difference being in the contributions to the r.h.s.\ of \eqref{delta-mu} from other critical points.

Namely, from \eqref{delta} we have
\begin{equation}\label{delta-mu-1}
\delta(\mu,\xi) =
	\left(\frac{\mu-\mu_0}{\mu-\mu_0^{-1}}\right)^{\ii h_0}
	\left(\frac{\mu+\mu_0^{-1}}{\mu+\mu_0}\right)^{\ii h_0}
	\left(\frac{\mu-\mu_1^{-1}}{\mu+\mu_1^{-1}}\right)^{\ii h_1}
	\left(\frac{\mu+\mu_1}{\mu_1-\mu}\right)^{\ii h_1}
	\eul^{\chi(\mu)},
\end{equation}
where $h_j=-\frac{1}{2\pi}\ln(1-|r(\mu_j)|^2)$, $j=0,1$ and
\begin{align}\label{chi-1}
\chi(\mu) &= \frac{1}{2\pi\ii}\left\{-\int_{-\infty}^{-\mu_1}\ln(\mu-s)\dd \ln(1-|r(s)|^2)+\left(\int_{-\mu_0}^{-\mu_0^{-1}}+ \int_{\mu_0^{-1}}^{\mu_0}\right) \ln\frac{1-|r(s)|^2}{1-|r(\mu_0)|^2}\frac{\dd s}{s-\mu}
\right.\notag \\
&\quad\qquad\left. +\int_{-\mu_1^{-1}}^{\mu_1^{-1}}\ln\frac{1-|r(s)|^2}{1-|r(\mu_1)|^2}\frac{\dd s}{s-\mu}-\int_{\mu_1}^{+\infty}\ln(s-\mu)\dd \ln(1-|r(s)|^2)\right\}.
\end{align}
Thus, using $\kappa_0(\mu_0)$, $f_0(\kappa_0(\mu_0))$, (see \eqref{kappa-mu}, \eqref{f0-hat-theta}), and similarly for $\kappa_1(\mu_1)$ and $f_1(\kappa_1(\mu_1))$
\[
\delta(\mu,\xi)\approx \hat\mu^{\ii h_0}(128 f_0\kappa_0^2 t)^{-\frac{\ii h_0}{2}}\left(\frac{\kappa_1 + \kappa_0}{\kappa_1 - \kappa_0}\right)^{\ii h_1}\eul^{\chi(\mu_0)}\text{ with }
\hat\mu=(\mu-\mu_0)\left(1+\frac{1}{\mu_0^{2}}\right)\sqrt{2f_0 t}.
\]
for $\mu$ near $\mu_0$ and
\[
\delta(\mu,\xi)\approx \hat\mu^{\ii h_1}(-128 f_1\kappa_1^2 t)^{-\frac{\ii h_1}{2}}\left(\frac{\kappa_1 + \kappa_0}{\kappa_1 - \kappa_0}\right)^{\ii h_0}\eul^{\chi(-\mu_1)}\text{ with }
\hat\mu=(\mu+\mu_1)\left(1+\frac{1}{\mu_1^{2}}\right)\sqrt{-2f_1 t}
\]
for $\mu$ near $-\mu_1$ (notice that $f_0(\kappa_0)=\frac{\kappa_0(3-4\kappa_0^2)}{(1+4\kappa_0^2)^3}>0$ whereas $f_1(\kappa_1)=\frac{\kappa_1(3-4\kappa_1^2)}{(1+4\kappa_1^2)^3}<0$). Consequently, the coefficients $\delta_{\mu_0}(\xi,t)$ and $\delta_{\mu_1}(\xi,t)$ to be used in the construction of $m_0$ \eqref{m0} for $\mu$ near $\mu_0$ and $-\mu_1$, respectively, are as follows:
\begin{equation}\label{d-mu-01}
\begin{split}
\delta_{\mu_0}(\xi,t) & =
\eul^{-\ii t \hat\theta(\kappa_0)}
\eul^{\chi(\mu_0)}
\left(\frac{\kappa_1 + \kappa_0}{\kappa_1 - \kappa_0}\right)^{\ii h_1}
(128 f_0\kappa_0^2(\mu_0) t)^{-\frac{\ii h_0}{2}},\\
\delta_{\mu_1}(\xi,t) & =
\eul^{\ii t \hat\theta(\kappa_1)}
\eul^{\chi(-\mu_1)}
\left(\frac{\kappa_1 + \kappa_0}{\kappa_1 - \kappa_0}\right)^{\ii h_0}
(-128 f_1\kappa_1^2(\mu_1) t)^{-\frac{\ii h_1}{2}},
\end{split}
\end{equation}
which implies (cf. \eqref{m0-as})
\begin{alignat*}{2}
m_0^{-1}(y,t,\mu) 
&=I + \frac{B_{\mu_0}(\xi,t)}{\sqrt{t}(\mu-\mu_0)}+\ord(t^{-1}),
&\quad&\mu\text{ inside }\gamma_0,\\
m_0^{-1}(y,t,\mu) 
&=I+\frac{B_{\mu_1}(\xi,t)}{\sqrt{t}(\mu+\mu_1)}+\ord(t^{-1}),
&&\mu\text{ inside }-\gamma_1,
\end{alignat*}
where (cf.\eqref{B})
\[
B_{\mu_0}(\xi,t)=\begin{pmatrix}
0 & B_0(\xi,t) \\ \bar  B_0(\xi,t) & 0
\end{pmatrix},\quad
B_{\mu_1}(\xi,t)=\begin{pmatrix}
0 & B_1(\xi,t) \\ \bar  B_1(\xi,t) & 0
\end{pmatrix},
\]
with
\begin{equation}\label{B-mu-01}
\begin{split}
B_0(\xi,t)&=\left(\frac{\kappa_1+\kappa_0}{\kappa_1-\kappa_0}\right)^{2\ii h_1}\frac{\ii\delta_{\mu_0}^2(\xi,t)\beta_{\mu_0}(\xi)}{(1+\mu_0^{-2})\sqrt{2f_0(\kappa_0)}},\\
B_1(\xi,t)&=\left(\frac{\kappa_1+\kappa_0}{\kappa_1 - \kappa_0}\right)^{2\ii h_0}\frac{\ii \delta_{\mu_1}^2(\xi,t)\beta_{\mu_1}(\xi)}{(1+\mu_1^{-2})\sqrt{-2 f_1(\kappa_1)}}.
\end{split}
\end{equation}
Here $\beta_{\mu_0}(\xi)$ is given by \eqref{beta} and
\[
\beta_{\mu_1}(\xi)=\sqrt{h_1}\eul^{\ii\left(\frac{\pi}{4}-\arg (-\bar r(-\mu_1))+\arg\Gamma(\ii h_1)\right)}.
\]
In turn, due to the symmetries, the asymptotics for $\hat m(y,t,0)$, $\hat m(y,t,\ii)$, and $\hat m_1(y,t)$ (and thus for $M^{R}(y,t,0)$,  $M^{R}(y,t,\ii)$, and $M^{R}_1(y,t)$) in the present case (cf. \eqref{mu0-as}-\eqref{mu1-as} and \eqref{M-R-as}) involve two terms:
\begin{align}\label{M-R-as-1}
M^{R}(y,t,0) &=I+\frac{4\ii}{\sqrt{t}}
\left(
\frac{\Im B_0(\xi,t)}{\mu_0} -\frac{\Im B_1(\xi,t)}{\mu_1}
\right)\begin{pmatrix}
	0 & 1 \\ -1 & 0
	\end{pmatrix}
	+ \osmall(t^{-1/2}), \notag\\
M^{R}(y,t,\ii) &
= \left(I+\frac{2\ii}{\sqrt{t}}
\left(
\frac{\Im B_0(\xi,t)}{\mu_0} -\frac{\Im B_1(\xi,t)}{\mu_1}
\right)
 \begin{pmatrix}
	0 & 1 \\ -1 & 0
\end{pmatrix}\right)\delta^{\sigma_3}(\ii,\xi)+ \osmall(t^{-1/2}), \\
M^{R}_1(y,t) &  =
\frac{4}{\sqrt{t}}\begin{pmatrix}
0 & \Re\frac{B_0}{(\mu_0-\ii)^2} +\Re\frac{B_1}{(\mu_1{+}\ii)^2}\\
\Re\frac{\bar B_0}{(\mu_0-\ii)^2} + \Re\frac{\bar B_1}{(\mu_1{+}\ii)^2} & 0\end{pmatrix}\delta^{\sigma_3}(\ii,\xi) + \osmall(t^{-1/2}),\notag
\end{align}
where  $\delta(\ii,\xi)$ is now given by
\begin{equation}\label{delta-i-1}
\delta(\ii,\xi)=\exp\left\{\frac{1}{\pi}\left(\int_{0}^{\mu_1^{-1}}+ \int^{\mu_0}_{\mu_0^{-1}}+\int_{\mu_1}^{+\infty}\right)\frac{\ln(1-|r(s)|^2)}{s^2+1}\dd s\right\}.
\end{equation}

It follows that the asymptotics for the parametric representation
of $\tilde u$, see \eqref{u-y-as} and \eqref{x-y-as},
takes the form
\begin{subequations}\label{u-x-as-1}
\begin{alignat}{3}
\label{u-y-as-1}
\hat u(y,t) &=\frac{8}{\sqrt{t}}\left(
\frac{(1-\mu_0^2)}{(1+\mu_0^2)^2}\Re B_0 + \frac{(1-\mu_1^2)}{(1+\mu_1^2)^2}\Re B_1
\right) +\osmall(t^{-1/2}),\\
\label{x-y-as-1}
x(y,t) & =
y+y_{01}(\xi) +\ord(t^{-1/2}),
\end{alignat}
\end{subequations}
where
$y_{01}(\xi)=\frac{2}{\pi}\left(\int_{0}^{\mu_1^{-1}} + \int^{\mu_0}_{\mu_0^{-1}}+\int_{\mu_1}^{+\infty}\right)\frac{\ln(1-|r(s)|^2)}{s^2+1}\dd s$.

Recalling the definitions \eqref{B-mu-01} of $B_j$, $j=0,1$, and arguing as in the case $0<\xi<2$, we arrive at the asymptotics of $\hat u(y,t)$ (cf.~\eqref{as-hat-u})
\begin{equation}\label{as-hat-u-1}
\hat u(y,t)=\sum_{j=0,1}\frac{C_1^{(j)}(\xi)}{\sqrt{t}}\cos\left\{C_2^{(j)}(\xi) t + C_3^{(j)}(\xi)\ln t + C_4^{(j)}(\xi)\right\} +\osmall(t^{-1/2}),
\end{equation}
where
\begin{subequations}\label{C1-4-01}
\begin{alignat}{3}
\label{C1-01}
C_1^{(j)}(\xi)&= -\left(\frac{8h_j\kappa_j}{|3-4\kappa_j^2|}\right)^{\frac{1}{2}},\\
\label{C2-01}
C_2^{(j)}(\xi)&= \frac{(-1)^j32\kappa_j^3}{(1+4\kappa_j^2)^2}, \\
\label{C3-01}
C_3^{(j)}(\xi)&= -h_j, \\
\label{C4-01}
C_4^{(j)}(\xi)&=\frac{3\pi}{4}-2\ii\chi((-1)^j\mu_j)-h_j\ln\frac{128\kappa_j^3|3-4\kappa_j^2|}{(1+4\kappa_j^2)^3}-\arg(-\bar r((-1)^j\mu_j))+\arg\Gamma(\ii h_j)\notag \\
&\quad + 2h_{1-j}\ln\frac{\kappa_1+\kappa_0}{\kappa_1-\kappa_0},
\end{alignat}
\end{subequations}
and $\chi(\mu)$ is given by \eqref{chi-1}.

Returning to the $(x,t)$ variables, $C_4^{(j)}(\xi)$, $j=0,1$ are to be replaced, similarly to \eqref{C4-1}, by
\begin{equation}\label{C4-1-01}
\tilde C_4^{(j)}(\zeta)=C_4^{(j)}(\zeta) +\frac{(-1)^j 4\kappa_j(\zeta)}{\pi}\left(\int_{0}^{\mu_1^{-1}}+\int_{\mu_0^{-1}}^{\mu_0}+\int_{\mu_1}^{+\infty}\right)\frac{\ln(1-|r(s)|^2)}{s^2+1}\dd s,
\end{equation}
which finally leads us to

\begin{theorem}\label{th3}
In the solitonless case, the solution $\tilde u(x,t)$ of the Cauchy problem for the mCH equation in the form \eqref{mCH2} has the following large-time asymptotics along the rays $\frac{x}{t}\eqqcolon\zeta$ in the sector of the $(x,t)$ half-plane $-\frac{1}{4}<\zeta<0$:
\[
\tilde u(x,t)=\sum_{j=0,1}\frac{C_1^{(j)}(\zeta)}{\sqrt{t}}\cos\left\{
C_2^{(j)}(\zeta) t + C_3^{(j)}(\zeta)\ln t + \tilde C_4^{(j)}(\zeta)\right\}+\osmall(t^{-1/2})
\]
with an error term uniform in any sector $-\frac{1}{4}+\varepsilon<\zeta<-\varepsilon$ where $\varepsilon$ is a small positive number. The coefficients $C_1^{(j)},C_2^{(j)},C_3^{(j)}$ are defined by \eqref{C1-01}-\eqref{C3-01} and $\tilde C_4^{(j)}$ is defined by \eqref{C4-1-01}-\eqref{C4-01}. In these definitions
\[
h_j=-\frac{1}{2\pi}\ln(1-|r(\mu_j)|^2),\quad
\kappa_0(\zeta)=\left(\frac{\sqrt{1+4\zeta}-1-\zeta}{4\zeta}\right)^{\frac{1}{2}},\quad
\kappa_1(\zeta)=\left(-\frac{\sqrt{1+4\zeta}+1+\zeta}{4\zeta}\right)^{\frac{1}{2}},
\]
and $\mu_j(\zeta)>1$, $j=0,1$ is characterized by the relation $\kappa_j(\zeta)=\frac{1}{4}(\mu_j(\zeta)-\mu_j(\zeta)^{-1})$.
\end{theorem}

Using again \eqref{utilde} we obtain, as a corollary, the large-time asymptotics of $u(x,t)$ in the sector $\frac{3}{4}<\frac{x}{t}<1$.

\begin{theorem}[$2^{\text{nd}}$ oscillatory region]\label{th4}
In the solitonless case, the solution $u(x,t)$ of the Cauchy problem \eqref{mCH1-ic} for the mCH equation has the following large-time asymptotics along the rays $\frac{x}{t}\eqqcolon\zeta$ in the sector of the $(x,t)$ half-plane defined by $\frac{3}{4}<\zeta<1$:
\[
u(x,t)=1+\sum_{j=0,1}\frac{C_1^{(j)}(\zeta-1)}{\sqrt{t}}\cos\left\{C_2^{(j)}(\zeta-1) t + C_3^{(j)}(\zeta-1)\ln t + \tilde C_4^{(j)}(\zeta-1)\right\}+\osmall(t^{-1/2}).
\]
The error term is uniform in any sector $\frac{3}{4}+\varepsilon<\zeta<1-\varepsilon$ where $\varepsilon$ is small and positive.
\end{theorem}

\begin{figure}[ht]
\centering\includegraphics[scale=.9]{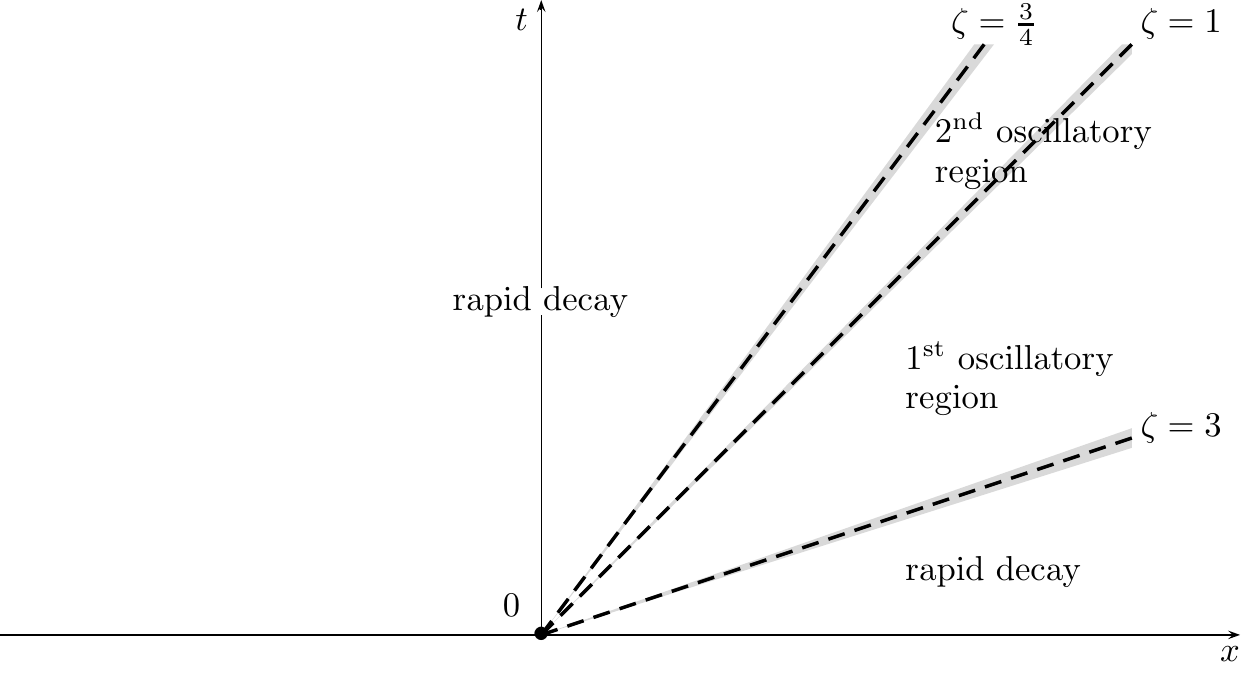}
\caption{Asymptotics for $u(x,t)$ according to $\zeta\coloneqq\frac{x}{t}$: the four regions.}
\label{f3}
\end{figure}

\begin{remark}
In the solitonless case, $u(x,t)$ decays rapidly to $0$ in the sectors $\frac{x}{t}>3$ and $\frac{x}{t}<\frac{3}{4}$, cf.~\cite{BS08-2}. This is due to the fact that for these ranges of values of $\frac{x}{t}$, $\theta(\mu,\xi)$ has no real stationary points (lying on the contour of the original RH problem).
\end{remark}

\begin{remark}
Transitions between the sectors (i.e., for $\frac{x}{t}$ near $\frac{3}{4}$ and $3$) are characterized by the merging of real stationary points of $\theta(\mu,\xi)$, which requires the use of different scalings of the spectral parameter. In analogy with the case of the Camassa--Holm equation (see \cite{BIS10}), one can expect that the asymptotics in the transition zones can be given in terms of Painlev\'e transcendents \cite{BKS20-2}.
\end{remark}



\begin{bibdiv}
\begin{biblist}
\bib{AK18}{article}{
   author={Anco, Stephen},
   author={Kraus, Daniel},
   title={Hamiltonian structure of peakons as weak solutions for the
   modified Camassa-Holm equation},
   journal={Discrete Contin. Dyn. Syst.},
   volume={38},
   date={2018},
   number={9},
   pages={4449--4465},
}
\bib{BIS10}{article}{
   author={Boutet de Monvel, Anne},
   author={Its, Alexander},
   author={Shepelsky, Dmitry},
   title={Painlev\'{e}-type asymptotics for the Camassa--Holm equation},
   journal={SIAM J. Math. Anal.},
   volume={42},
   date={2010},
   number={4},
   pages={1854--1873},
}
\bib{BKS20}{article}{
   author={Boutet de Monvel, Anne},
   author={Karpenko, Iryna},
   author={Shepelsky, Dmitry},
   title={A Riemann-Hilbert approach to the modified Camassa--Holm equation
   with nonzero boundary conditions},
   journal={J. Math. Phys.},
   volume={61},
   date={2020},
   number={3},
   pages={031504, 24},
}
\bib{BKS20-2}{article}{
   author={Boutet de Monvel, Anne},
   author={Karpenko, Iryna},
   author={Shepelsky, Dmitry},
   title={Painlev\'{e}-type asymptotics for the modified Camassa--Holm equation},
   status={in preparation},
}
\bib{BKST09}{article}{
   author={Boutet de Monvel, Anne},
   author={Kostenko, Aleksey},
   author={Shepelsky, Dmitry},
   author={Teschl, Gerald},
   title={Long-time asymptotics for the Camassa--Holm equation},
   journal={SIAM J. Math. Anal.},
   volume={41},
   date={2009},
   number={4},
   pages={1559--1588},
}
\bib{BS08}{article}{
   author={Boutet de Monvel, Anne},
   author={Shepelsky, Dmitry},
   title={Riemann-Hilbert problem in the inverse scattering for the
   Camassa--Holm equation on the line},
   conference={
      title={Probability, geometry and integrable systems},
   },
   book={
      series={Math. Sci. Res. Inst. Publ.},
      volume={55},
      publisher={Cambridge Univ. Press, Cambridge},
   },
   date={2008},
   pages={53--75},
}
\bib{BS08-2}{article}{
   author={Boutet de Monvel, Anne},
   author={Shepelsky, Dmitry},
   title={Long-time asymptotics of the Camassa--Holm equation on the line},
   conference={
      title={Integrable systems and random matrices},
   },
   book={
      series={Contemp. Math.},
      volume={458},
      publisher={Amer. Math. Soc., Providence, RI},
   },
   date={2008},
   pages={99--116},
}
\bib{BS09}{article}{
   author={Boutet de Monvel, Anne},
   author={Shepelsky, Dmitry},
   title={Long time asymptotics of the Camassa--Holm equation on the
   half-line},
   journal={Ann. Inst. Fourier (Grenoble)},
   volume={59},
   date={2009},
   number={7},
   pages={3015--3056},
}
\bib{CH93}{article}{
   author={Camassa, Roberto},
   author={Holm, Darryl D.},
   title={An integrable shallow water equation with peaked solitons},
   journal={Phys. Rev. Lett.},
   volume={71},
   date={1993},
   number={11},
   pages={1661--1664},
}
\bib{CHH94}{article}{
   author={Camassa, Roberto},
   author={Holm, Darryl D.},
   author={Hyman, James M.},
   title={A new integrable shallow water equation},
   journal={Adv. Appl. Mech.},
   volume={31},
   date={1994},
   number={1},
   pages={1--33},
}
\bib{CS17}{article}{
   author={Chang, Xiangke},
   author={Szmigielski, Jacek},
   title={Liouville integrability of conservative peakons for a modified CH
   equation},
   journal={J. Nonlinear Math. Phys.},
   volume={24},
   date={2017},
   number={4},
   pages={584--595},
}
\bib{CS18}{article}{
   author={Chang, Xiangke},
   author={Szmigielski, Jacek},
   title={Lax integrability and the peakon problem for the modified
   Camassa-Holm equation},
   journal={Comm. Math. Phys.},
   volume={358},
   date={2018},
   number={1},
   pages={295--341},
}
\bib{CGLQ16}{article}{
   author={Chen, Robin Ming},
   author={Guo, Fei},
   author={Liu, Yue},
   author={Qu, Changzheng},
   title={Analysis on the blow-up of solutions to a class of integrable
   peakon equations},
   journal={J. Funct. Anal.},
   volume={270},
   date={2016},
   number={6},
   pages={2343--2374},
}
\bib{CLQZ15}{article}{
   author={Chen, Robin Ming},
   author={Liu, Yue},
   author={Qu, Changzheng},
   author={Zhang, Shuanghu},
   title={Oscillation-induced blow-up to the modified Camassa-Holm equation
   with linear dispersion},
   journal={Adv. Math.},
   volume={272},
   date={2015},
   pages={225--251},
}
\bib{CL09}{article}{
   author={Constantin, Adrian},
   author={Lannes, David},
   title={The hydrodynamical relevance of the Camassa--Holm and
   Degasperis-Procesi equations},
   journal={Arch. Ration. Mech. Anal.},
   volume={192},
   date={2009},
   number={1},
   pages={165--186},
}
\bib{DZ93}{article}{
   author={Deift, P.},
   author={Zhou, X.},
   title={A steepest descend method for oscillatory 
   Riemann--Hilbert problems. Asymptotics for the MKdV equation},
   journal={Ann. Math.},
   volume={137},
   date={1993},
   number={2},
   pages={295--368},
}
\bib{E19}{article}{
   author={Eckhardt, Jonathan},
   title={Unique solvability of a coupling problem for entire functions},
   journal={Constr. Approx.},
   volume={49},
   date={2019},
   number={1},
   pages={123--148},
}
\bib{ET13}{article}{
   author={Eckhardt, Jonathan},
   author={Teschl, Gerald},
   title={On the isospectral problem of the dispersionless Camassa-Holm
   equation},
   journal={Adv. Math.},
   volume={235},
   date={2013},
   pages={469--495},
}
\bib{ET16}{article}{
   author={Eckhardt, Jonathan},
   author={Teschl, Gerald},
   title={A coupling problem for entire functions and its application to the
   long-time asymptotics of integrable wave equations},
   journal={Nonlinearity},
   volume={29},
   date={2016},
   number={3},
   pages={1036--1046},
   issn={0951-7715},
   review={\MR{3465992}},
   doi={10.1088/0951-7715/29/3/1036},
}
\bib{F95}{article}{
   author={Fokas, A. S.},
   title={On a class of physically important integrable equations},
   note={The nonlinear Schr\"{o}dinger equation (Chernogolovka, 1994)},
   journal={Phys. D},
   volume={87},
   date={1995},
   number={1-4},
   pages={145--150},
}
\bib{FGLQ13}{article}{
   author={Fu, Ying},
   author={Gui, Guilong},
   author={Liu, Yue},
   author={Qu, Changzheng},
   title={On the Cauchy problem for the integrable modified Camassa-Holm
   equation with cubic nonlinearity},
   journal={J. Differential Equations},
   volume={255},
   date={2013},
   number={7},
   pages={1905--1938},
}
\bib{Fu96}{article}{
   author={Fuchssteiner, Benno},
   title={Some tricks from the symmetry-toolbox for nonlinear equations:
   generalizations of the Camassa--Holm equation},
   journal={Phys. D},
   volume={95},
   date={1996},
   number={3-4},
   pages={229--243},
}
\bib{GL18}{article}{
   author={Gao, Yu},
   author={Liu, Jian-Guo},
   title={The modified Camassa-Holm equation in Lagrangian coordinates},
   journal={Discrete Contin. Dyn. Syst. Ser. B},
   volume={23},
   date={2018},
   number={6},
   pages={2545--2592},
}
\bib{GLOQ13}{article}{
   author={Gui, Guilong},
   author={Liu, Yue},
   author={Olver, Peter J.},
   author={Qu, Changzheng},
   title={Wave-breaking and peakons for a modified Camassa-Holm equation},
   journal={Comm. Math. Phys.},
   volume={319},
   date={2013},
   number={3},
   pages={731--759},
}
\bib{HFQ17}{article}{
   author={Hou, Yu},
   author={Fan, Engui},
   author={Qiao, Zhijun},
   title={The algebro-geometric solutions for the Fokas-Olver-Rosenau-Qiao
   (FORQ) hierarchy},
   journal={J. Geom. Phys.},
   volume={117},
   date={2017},
   pages={105--133},
}
\bib{IU86}{article}{
   author={Its, A. R.},
   author={Ustinov, A. F.},
   title={Time asymptotics of the solution of the Cauchy problem for the
   nonlinear Schr\"{o}dinger equation with boundary conditions of finite density
   type},
   language={Russian},
   journal={Dokl. Akad. Nauk SSSR},
   volume={291},
   date={1986},
   number={1},
   pages={91--95},
   translation={
      journal={Soviet Phys. Dokl.},
      volume={31},
      date={1986},
      number={11},
      pages={893–-895},
   },
}
\bib{J02}{article}{
   author={Johnson, R. S.},
   title={Camassa--Holm, Korteweg--de Vries and related models for water
   waves},
   journal={J. Fluid Mech.},
   volume={455},
   date={2002},
   pages={63--82},
}
\bib{K16}{article}{
   author={Kang, Jing},
   author={Liu, Xiaochuan},
   author={Olver, Peter J.},
   author={Qu, Changzheng},
   title={Liouville correspondence between the modified KdV hierarchy and
   its dual integrable hierarchy},
   journal={J. Nonlinear Sci.},
   volume={26},
   date={2016},
   number={1},
   pages={141--170},
}
\bib{L15}{article}{
   author={Lenells, Jonatan},
   title={The nonlinear steepest descent method for Riemann-Hilbert problems
   of low regularity},
   journal={Indiana Univ. Math. J.},
   volume={66},
   date={2017},
   number={4},
   pages={1287--1332},
}
\bib{LLOQ14}{article}{
   author={Liu, Xiaochuan},
   author={Liu, Yue},
   author={Olver, Peter J.},
   author={Qu, Changzheng},
   title={Orbital stability of peakons for a generalization of the modified
   Camassa-Holm equation},
   journal={Nonlinearity},
   volume={27},
   date={2014},
   number={9},
   pages={2297--2319},
}
\bib{LOQZ14}{article}{
   author={Liu, Yue},
   author={Olver, Peter J.},
   author={Qu, Changzheng},
   author={Zhang, Shuanghu},
   title={On the blow-up of solutions to the integrable modified
   Camassa-Holm equation},
   journal={Anal. Appl. (Singap.)},
   volume={12},
   date={2014},
   number={4},
   pages={355--368},
}
\bib{MN02}{article}{
   author={Mikhailov, A. V.},
   author={Novikov, V. S.},
   title={Perturbative symmetry approach},
   journal={J. Phys. A},
   volume={35},
   date={2002},
   number={22},
   pages={4775--4790},
}
\bib{N09}{article}{
   author={Novikov, Vladimir},
   title={Generalizations of the Camassa--Holm equation},
   journal={J. Phys. A},
   volume={42},
   date={2009},
   number={34},
   pages={342002, 14},
}
\bib{OR96}{article}{
   author={Olver, P. J.},
   author={Rosenau, P.},
   title={Tri-hamiltonian duality between solitons and solitary-wave
solutions having compact support},
   journal={Phys. Rev. E},
   volume={53},
   date={1996},
   number={2},
   pages={1900},
}
\bib{Q06}{article}{
   author={Qiao, Zhijun},
   title={A new integrable equation with cuspons and W/M-shape-peaks
   solitons},
   journal={J. Math. Phys.},
   volume={47},
   date={2006},
   number={11},
   pages={112701, 9},
}
\bib{QLL13}{article}{
   author={Qu, Changzheng},
   author={Liu, Xiaochuan},
   author={Liu, Yue},
   title={Stability of peakons for an integrable modified Camassa-Holm
   equation with cubic nonlinearity},
   journal={Comm. Math. Phys.},
   volume={322},
   date={2013},
   number={3},
   pages={967--997},
}
\bib{S96}{article}{
   author={Schiff, Jeremy},
   title={Zero curvature formulations of dual hierarchies},
   journal={J. Math. Phys.},
   volume={37},
   date={1996},
   number={4},
   pages={1928--1938},
}
\bib{V00}{article}{
   author={Vartanian, A. H.},
   title={Large-time continuum asymptotics of dark solitons},
   journal={Inverse Problems},
   volume={16},
   date={2000},
   number={4},
   pages={L39--L46},
}
\bib{V03}{article}{
   author={Vartanian, A. H.},
   title={Exponentially small asymptotics of solutions to the defocusing
   nonlinear Schr\"{o}dinger equation},
   journal={Appl. Math. Lett.},
   volume={16},
   date={2003},
   number={3},
   pages={425--434},
}
\bib{WLM20}{article}{
   author={Wang, Gaihua},
   author={Liu, Q.P.},
   author={Mao, Hui},
   title={The modified Camassa-Holm equation: Bäcklund transformation and
   nonlinear superposition formula},
   journal={J. Phys. A},
   volume={53},
   date={2020},
   number={29},
   pages={294003--294018},
}
\bib{YQZ18}{article}{
   author={Yan, Kai},
   author={Qiao, Zhijun},
   author={Zhang, Yufeng},
   title={On a new two-component $b$-family peakon system with cubic
   nonlinearity},
   journal={Discrete Contin. Dyn. Syst.},
   volume={38},
   date={2018},
   number={11},
   pages={5415--5442},
}
\end{biblist}
\end{bibdiv}
\end{document}